    \documentclass[11pt,a4paper]{article}
\usepackage[dvips]{graphics,graphicx}
\usepackage{amsfonts,amssymb,amsmath,color,mathrsfs, amstext}
\usepackage{amsbsy, amsopn, amscd, amsxtra, amsthm,authblk,enumerate,algorithmicx,algorithm}
\usepackage{bm,bbm}
\usepackage{algpseudocode}
\usepackage{upref}
\usepackage{geometry}
\usepackage{ulem}
\usepackage[displaymath]{lineno}
\usepackage{float}
\usepackage{yhmath}
\usepackage{amsmath}
\usepackage{amssymb}
\usepackage{amstext}
\usepackage{mathrsfs}
\usepackage{a4wide}
\usepackage{graphicx}
\usepackage{bbm}
\usepackage{verbatim}
\usepackage{bm}
\usepackage{relsize} 
\usepackage{hyperref}
\allowdisplaybreaks \numberwithin{equation}{section}
\usepackage{color}
\usepackage{cases}
\hypersetup{colorlinks=true}
\hypersetup{
    colorlinks=true,
    linkcolor=blue,
    filecolor=blue,
    urlcolor=blue,
    anchorcolor=black,
    citecolor={red},
}

\numberwithin{equation}{section}

\newtheorem{theorem}{Theorem}[section]

\newtheorem{lemma}[theorem]{Lemma}

\theoremstyle{definition}

\theoremstyle{remark}

\newcommand{\norm}[1]{\left\Vert#1\right\Vert}
\newcommand{\abs}[1]{\left\vert#1\right\vert}

\allowdisplaybreaks

\begin{document}

\title{\bf{Existence of stationary vortex patches for the \lowercase{g}SQG in bounded domains}\vspace{0.5cm}}

\author{\large{Vladimir Angulo-Castillo}\thanks{Department of Mathematics, National University of Colombia, Kilometro 9 vía a Caño Limón, Arauca, Colombia. E-mail: vlcastillo@unal.edu.co} , \large{ Edison Cuba}\thanks{IMECC-Department of Mathematics, University of Campinas, Rua S\'{e}rgio Buarque de Holanda, 651, 13083-859, Campinas, SP, Brazil. E-mail: ecubah@ime.unicamp.br} , \ and \large{ Lucas C. F. Ferreira}\thanks{IMECC-Department of Mathematics, University of Campinas, Rua S\'{e}rgio Buarque de Holanda, 651, 13083-859, Campinas, SP, Brazil. E-mail: lcff@ime.unicamp.br (corresponding author)}}
\date{}
\maketitle


\vspace{-0.7cm}

\begin{abstract}
	In this paper we  show the existence of time-periodic %
    vortex patches  for the genera\-lized surface quasi-geostrophic %
    equation   within a bounded domain. This construction is %
    carried out for values of $\gamma$ in the range of $(1,2)$. %
    The resulting vortex patches possess a fixed vorticity and %
    total flux, and they are located in the neighborhood of %
    critical points that are non-degenerate for the %
    Kirchhoff--Routh equation. The proof is accomplished through %
    a combination of analyzing the linearization of the %
    contour dynamics equation and employing the implicit %
    function theorem as well as carefully selected function spaces.



\end{abstract}
	

\section{Introduction}
In this study, we delve into a category of active scalar systems that interact with an incompressible flow within a two-dimensional framework. To be precise, through an examination of the contour dynamics equation and the application of the implicit function theorem, we establish the presence of stationary vortex patches. These patches are characterized by a constant total flux and a fixed vorticity for each individual patch within the framework of the generalized surface quasi-geostrophic (gSQG) equation, which are defined within a bounded domain. In that sense, this model explains how the potential temperature $\omega$ evolves under the influence of the transport equation
\begin{equation}\label{1-1}
	\begin{cases}
		\partial_t\omega+\mathbf{v}\cdot \nabla \omega =0&\text{in}\ \Omega\times (0,T),\\
		\ \mathbf{v}=\nabla^\perp(-\Delta)^{-1+\frac{\gamma}{2}}\omega     &\text{in}\ \Omega\times (0,T),\\
		\omega\big|_{t=0}=\omega_0 &\text{in}\ \Omega,\\
	\end{cases}
\end{equation}
where $\Omega$ is a bounded domain in two-dimensional space, and we consider a parameter $\gamma$ satisfying the condition $0\leq\gamma<2$. The variable $\omega(\boldsymbol{x},t)$, defined for $\boldsymbol{x}$ within $\Omega$ and $t$ in the interval $(0, T)$, represents an active scalar being advected by a velocity field $\mathbf{v}(\boldsymbol{x},t)$. This velocity field is generated by $\omega$, and $\nabla^{\perp}=(\partial_{2},-\partial_{1})$. The operator denoted by $(-\Delta)^{-1+\frac{\gamma}{2}}$ is defined as
\begin{equation*}
(-\Delta)^{-1+\frac{\gamma}{2}}\omega(\boldsymbol{x})=\int_{\Omega} K_{\gamma}(\boldsymbol{x},\boldsymbol{y})\omega(\boldsymbol{y}),d\boldsymbol{y},
\end{equation*}
where the term $K_{\gamma}(\boldsymbol{x},\boldsymbol{y})$ represents the Green function associated with the fractional Laplacian in bounded domains with smooth boundaries. It is defined for each pair of points $\boldsymbol{x},\boldsymbol{y}\in\Omega$, where $\boldsymbol{x}\neq \boldsymbol{y}$, as follows:
\begin{equation*}
	K_\gamma(\boldsymbol{x},\boldsymbol{y})=
	\begin{cases}
		-\frac{1}{2\pi}\log |\boldsymbol{x}-\boldsymbol{y}| + K_0^0(\boldsymbol{x},\boldsymbol{y}),  \ \  & \ \ \gamma=0,\\
		\frac{C_{\gamma}}{|\boldsymbol{x}-\boldsymbol{y}|^{\gamma}} + K_\gamma^0(\boldsymbol{x},\boldsymbol{y}),\ \  & \ \ \gamma\in(0,2),
	\end{cases}
\end{equation*}
with $\Gamma(\cdot)$ being the Euler gamma function and $C_{\gamma}=\frac{2^{\gamma-1}\Gamma (\frac{\gamma}{2})}{ \Gamma(1-\frac{\gamma}{2})}$. Additionally, $K_\gamma^0$ belongs to the class of infinitely differentiable functions $C^\infty (\Omega\times \Omega)$, as discussed in \cite[Lemma 2.3]{Hmidi}. \medskip

We highlight that the system (\ref{1-1}) covers the cases of the 2D incompressible Euler equations by taking $\gamma=0$ and the inviscid SQG equations when considering $\gamma=1$. The system (\ref{1-1}) for $0<\gamma<2$ was initially introduced by Córdoba \textit{et al.} for the flat case $\mathbb{R}^{2}$ in their work \cite{Cord}. Over the past decade, it has garnered significant attention and scrutiny as it represents a generalization of both the Euler equation and the SQG equation. Notice that the case $\gamma=2$ produces stationary solutions. \medskip

The equation governing the evolution of the interface is defined by a periodic curve $z(x,t)$ with a $2\pi$ period, which adheres to the formula
\begin{equation*}
\partial_{t}z(x,t)=\frac{C_{\gamma }}{2\pi }\int_{0}^{2\pi }\frac{\partial_{y}z(y,t)}{|z(x,t)-z(y,t)|^{\gamma }}\mathrm{d}y,
\end{equation*}
known in the literature as the contour dynamics equation. However, this integral becomes divergent for $\alpha \in \lbrack 1,2)$. To address this singularity, the equation for interface evolution can be reformulated as
\begin{equation*}
\partial_{t}z(x,t)=\frac{C_{\gamma }}{2\pi }\int_{0}^{2\pi }\frac{\partial
_{y}z(y,t)-\partial _{y}z(x,t)}{|z(x,t)-z(y,t)|^{\gamma }}\mathrm{d}y.
\end{equation*}
This additional term in the integral does not alter the evolution of the patch, as we can include terms in the tangential direction.

It is known that the SQG equation ($\gamma=1$ in (\ref{1-1})) serves as a model for tracking the evolution of temperature within a broader quasi-geostrophic system applicable to atmospheric flows near the tropopause \cite{Held,Juckes} and  ocean flows in the upper layers \cite{Lape}, see also \cite{ConsMT,Lape2017}. Furthermore, the SQG equation has gained substantial attention as a simplified representation of the three dimensional Euler equation (see \cite{ConsMT}). \medskip

Active scalar equations have recently attracted attention mainly because they can shed more light on the problem of the duality between the presence of solutions exhibiting singularities within a finite time frame and the global well-posedness of classical solutions (see \cite{Castro,Chae12,He,Kiselev,Kiselev1} and references therein). \medskip

In the context of determining the local well-posedness of solutions across the entire plane, numerous studies have explored various function spaces for initial data. For instance, we can reference the research conducted by Chae \textit{et al.} \cite{Chae12}, which demonstrates the local existence of patch-type solutions within the framework of Sobolev spaces. However, the question of global well-posedness for classical solutions in the entire space remains unresolved, except for the special case of Euler equation when $\gamma=0$ in equation (\ref{1-1}), as discussed in \cite{Elgindi,Majda,Yud}. This is due to the fact that for the general case with $0<\gamma<2$ the velocity is known to exhibit singularity and does not fit into the Lipschitz class. The local well-posedness of the problem (\ref{1-1}) for classical solutions was originally established by Constantin \textit{et al.} \cite{ConsMT} and later extended to initial data with sufficient regularity (see \cite{Chae12,Gancedo,Kiselev1}). Similarly, solutions in different function spaces have been found to exist locally (see \cite{Chae12,Gancedo,Kiselev1}). Furthermore, the global existence of weak solutions is well-established, primarily due to the work of Lazar and Zue \cite{Lazar}, Marchand \cite{Marchand}, and Resnick \cite{Resnick}. In addition, the papers by Buckmaster and Isett in \cite{Buckmaster,Isett} explore the multiplicity of weak solutions to the problem (\ref{1-1}) with $\gamma=1$. \medskip

On the other hand, another category of solutions that garners significant attention in academic discussions involves $\gamma$-patches. These $\gamma$-patches are solutions to equation (\ref{1-1}) and are obtained by using initial data in the form $\omega_0=\boldsymbol \chi_{D}$, where $D$ is a bounded domain within the two-dimensional real space, and $\boldsymbol \chi_{D}$ represents the characteristic function of that domain. Then, due to the transport equation $\partial_t\omega+\mathbf{v}\cdot \nabla \omega =0$, the solution temporarily assumes the shape $\omega_t=\boldsymbol \chi_{D_t}$ for a brief period. In this scenario, the patch structure remains unchanged for a brief period, while the boundary undergoes evolution in accordance with the appropriate contour dynamics equations, as referenced in \cite{Chae12,Gancedo,Rodrigo}. Some results in this direction were shown for the half space \cite{Gancedo1,Kiselev,Kiselev1} and for smooth bounded domains \cite{Hmidi,Kiselev2}. However, it is important to note that the global persistence of boundary regularity over time is only known to hold for the case when $\gamma=0$. This result is established in Chemin's paper \cite{Chemin} and is also presented with an alternative proof in the paper by Bertozzi and Constantin \cite{Bertozzi}. The interest in these solutions is, in part, driven by numerical and experimental observations that strongly indicate the emergence of singularities within a finite time period, as reported in references \cite{Cord,Scott,Scott1}. In this direction, we observe that Hassainia and Hmidi in \cite{Hassainia} constructed the pioneering example of ``V-States'' as non-trivial instances of global solutions in the vortex patch category. These solutions were established for values of $\gamma$ within the range $[0,1)$, and their construction involved the utilization of contour dynamics equations and bifurcation theory. Subsequently, Castro \textit{et al.} \cite{Cas1} extended this result to encompass cases where $\gamma$ falls within the interval $[1,2)$, and their work also delved into the examination of boundary regularity. In relation to the existence of vortex patch type solutions for gSQG equation, there are several works that address the problem in different contexts (see, for example, \cite{Cas1,Castro1,Cord0,delaHoz,Garcia,Godard,Gomez,Hassainia1,Hmidi0,Renault}).  In particular, Castro et al. in \cite{Castro2}, accomplished the task of establishing the global existence of smooth solutions to the problem described in (\ref{1-1}) by employing a bifurcation argument centered around a radially symmetric function. Next, Gravejat and Smets \cite{Gravejat} investigated the presence of smooth traveling vortex pairs for the problem outlined in (\ref{1-1}) with $\gamma=1$, obtaining a result that, thanks to Godard-Cadillac in \cite{Godard0}, was extended to any $\gamma\in[0,2)$. For their part, Kiselev \textit{et al.} \cite{Kiselev} obtained a result on the singularity with multi-signed patches in finite time for (\ref{1-1}) with the condition $0<\gamma<\frac{1}{12}$, and later improved by Gancedo \textit{et al.} \cite{Gancedo1} for $0<\gamma<\frac{1}{3}$.
\medskip

A key point in the development of our problem (\ref{1-1}) is the analysis of the problem of desingularization of point vortices, which is related to the search for concentrated global solutions. In this context, the work of Hmidi and Mateu in \cite{Hmidi0} is notable as it utilizes the contour dynamics equation and the implicit function theorem. In their research, they detail the construction of pairs of traveling and co-rotating patches characterized by substantial vorticity strength and a limited area. Hmidi and Mateu's approach was primarily applied to desingularize various phenomena, including the Karman Vortex Street studied by García \cite{Garcia0}, the Thomson polygon analyzed by both García \cite{Garcia} and Cao \cite{cao}, the pair of asymmetric vortices investigated by Hassainia and Hmidi \cite{Hassainia00}, and general steady configurations involving a finite number of point vortices with generalized surface quasi-geostrophic interactions, as explored by Hassainia and Wheeler \cite{Hassainia1}. Moreover, using the Lyapunov-Schmidt reduction method Cao \textit{et al.} in \cite{Ao} proved the presence of smooth concentrated solutions that travel and rotate in the context of the gSQG equation. \medskip

The study of SQG-type equations defined on bounded smooth domains is, in part, much more complicated than in the $\mathbb{R}^{2}$ case, since there the Green function cannot be expressed explicitly. In the context of these systems, it is worth highlighting the pioneering contributions of Constantin and Ignatova, as outlined in \cite{Constantin,Constantin1}, for \eqref{1-1} with $\gamma=1$  with critical dissipation. Subsequently, in \cite{Constantin2}, the authors demonstrated that for the critical dissipative SQG equation in bounded domains, global regularity extending up to the boundary is established if and only if a particular quantitative criterion related to the scalar function becoming zero at the boundary is met. This condition is connected to the square root of the Dirichlet Laplacian dissipation. Following these developments, Ignatova, in the work presented in \cite{Ignatova}, constructs global bounded interior Lipschitz solutions starting from arbitrary large bounded interior Lipschitz initial data. Utilizing a method grounded in De Giorgi's techniques, the research establishes global H\"{o}lder regularity that extends all the way to the boundary of the solution, as discussed in \cite{Stokols}. In relation to the general case that concerns our paper, we highlight that Constantin and Nguyen \cite{ConstantinN1} constructed global weak solutions in $L^{2}$ for the SQG case $\gamma=1$, result that was then generalized for any $\gamma$ in the interval $(1,2)$ with more singular constitutive law in the velocity (see \cite{Nguyen}). Moreover, local well-posedness for classical solutions of the inviscid SQG equation within bounded and smooth domains was established in \cite{ConstantinN}. \medskip

The aforementioned findings reveal that the behavior of the gSQG equation within bounded domains is quite complex when it comes to well-posedness and the formation of singularities. In fact, a comprehensive understanding of these properties is yet to be fully established for the entire range of parameter values $\alpha \in (0,2)$. This situation naturally encourages further exploration of solutions that exhibit specific characteristics and unique dynamics. In this context, there are not many studies in domains other than the plane in a noteworthy category of solutions that evolve from initial measures, particularly those associated with active scalar equations.  In the case of these problems, Hmidi \textit{et al.} in \cite{Hmidi} proved for the gSQG equation (\ref{1-1}) the existence of the V-states with $\gamma\in(0,1)$ in the unit disc, and then, Cao \textit{et al.} in \cite{cao1} demonstrated the existence of stationary vortex patches. These patches maintained both fixed vorticity and a consistent total flux for each patch. They achieved this within the context of the SQG equation, considering a general bounded domain. Their approach involved a desingularization method applied to point vortices. \medskip

The purpose of this paper is to establish a collection of stationary vortex patches within a bounded domain for the gSQG with $\gamma\in(1,2)$, achieved through the desingularization of point vortices.  This construction depends on the requirement that the Kirchhoff-Routh function possesses non-degenerate critical points. To the best of our knowledge, this represents the first result concerning the presence of a set of stationary patches for the gSQG equation within a bounded domain in this more singular parameter range. The inclusion of such a parameter range for $\gamma$ presents us with a more intricate challenge in establishing existence of stationary vortex patches for the gSQG in bounded domains. Our proof relies on a careful selection of the functional spaces (namely, $X^{k+\gamma-1}$, $Y^{k}$ and $Y_0^{k}$, see Section \ref{sub-2.1}, p.8), where we employ the implicit function theorem to contour dynamics equations for stationary vortex patches, inspired by contributions of the previous works \cite{cao}, \cite{cao1}, \cite{Cas1}, \cite{Hassainia}, \cite{Hmidi}. Hence, in order to ensure the existence of stationary vortex patches for \eqref{1-1} via the implicit function theorem, we need to obtain some properties for the functional $G_i(\varepsilon, \boldsymbol{\rho}, \boldsymbol{x}, \boldsymbol{g})$  defined in \eqref{funct} below. For example, we need to verify its $C^1$-regularity and compute the linearized operator $\partial_{\boldsymbol{g}}\boldsymbol{G}(0, \boldsymbol{\rho}, \boldsymbol{x}, \mathbf{0})$. However, we encounter two issues. First, this operator lacks invertibility from $(X^{k+\gamma-1})^m$ to $(Y^{k-1})^m$, but it possesses invertibility from $(X^{k+\gamma-1})^m$ to $(Y_0^{k-1})^m$. Moreover, it is crucial to consider an appropriate segment for obtaining the Fourier series of the linearization of $\boldsymbol{G}$ around $(0,  \boldsymbol{\rho}, \boldsymbol{x}_0, \mathbf{0})$. This consideration stems from our utilization of a complete Fourier series within the functional space to demonstrate that the linearization constitutes an isomorphism, as delineated in Lemma \ref{iso}.

For a collection of $m$ real numbers $\kappa_1, \kappa_2, \ldots, \kappa_m$, we establish the Kirchhoff-Routh function on $\Omega^m$ in the following manner
\begin{equation}\label{1-5}
    \mathcal{W}_m(x_1,x_2,..,x_m)=-\sum\limits_{i\neq j}^m\kappa_i\kappa_j K^1_\gamma(x_i,x_j)+\sum\limits_{i=1}^m\kappa_i^2K_\gamma^0(x_i,x_i),
\end{equation}
where $\Omega^m$ is the set of vectors $\boldsymbol{x} = (x_1, x_2, \ldots, x_m)$ such that each $x_i$ belongs to the set $\Omega$ for $i = 1, 2, \ldots, m$ and $K^1_\gamma(x,y)=\frac{C_{\gamma}}{|\boldsymbol{x}-\boldsymbol{y}|^{\gamma}}$.\medskip

It is established in  \cite{Lin1, Lin2} that the positioning of $m$ point vortices with strengths $\kappa_i$ ($i=1,\ldots,m$) within the domain $\Omega$ requires that these vortices are located at a critical point of the Kirchhoff-Routh function $\mathcal{W}_m$ that constitutes a non-degenerate critical point of $\mathcal{W}_m$. We are going to proceed to construct a family of vortex patches for small values of $\varepsilon$. In this sense, we construct a family of time-periodic vortex patches for the gSQG. \medskip

Now, we are in a position to give our main result.
\begin{theorem}\label{thm1}
Consider a bounded domain $\Omega\subset \mathbb{R}^2$ with a smooth boundary and $m$ given positive values $\kappa_i$ ($i=1,\ldots, m)$. Assume that $\boldsymbol{x}_0=(x_{0,1},\ldots,x_{0,m})\in \Omega^m$, with $x_{0,i}\not=x_{0,j}$ for $i\not=j$, is an isolated critical point of $\mathcal{W}_m$ as defined in \eqref{1-5} and satisfies the nondegeneracy condition: $\text{deg} \left(\nabla \mathcal{W}_m, \boldsymbol{x}_0\right)\not=0$. Under these conditions, there exists  $\varepsilon_0>0$ such that, for all $0<\varepsilon<\varepsilon_0$, a stationary vortex patch solution $\omega_{\varepsilon}$ can be constructed, which exhibits the following characteristics:
	\begin{itemize}
		\item[$(i)$]  $\omega_{\varepsilon}=\sum_{i=1}^m \frac{1}{\varepsilon^2} \chi_{\Gamma_i}$  within specific domains  $\Gamma_i \subset \Omega, i=1, \ldots, m$.

        \item[$(ii)$]   The boundaries $\partial \Gamma_i$ for $i=1, \ldots, m$  can be defined using the subsequent parameterization
     $$\partial \Gamma_i=\left\{x_{\varepsilon, i}+\varepsilon\left(\sqrt{\frac{\kappa_i}{\pi}}+o(1)\right)(\cos \beta, \sin \beta) \mid \beta \in[0,2 \pi)\right\},
        $$
        where $x_{\varepsilon, i}=x_{0, i}+o(1)$ as $\varepsilon \rightarrow 0$.

        \item[$(iii)$] The total flux for each patch remains fixed as
        $$
        \frac{1}{\varepsilon^2}\left|\Gamma_i\right|=\kappa_i, \quad \forall i=1, \ldots, m.
        $$

        \item[$(iv)$] As $\varepsilon \rightarrow 0^+$, one has the following convergence in the sense of measures
        $$
        \omega_{\varepsilon} \rightarrow \sum_{i=1}^n \delta\left(x-x_{0, i}\right) \text { weakly, }
        $$
        where $\delta\left(x-x_{0, i}\right)$ represents the Dirac delta function concentrates at  the point $x_{0, i}$.

		\item[$(v)$] The interior of each domain $\Gamma_i$ is convex, for every $i=1,\ldots,m$.
	\end{itemize}
\end{theorem}
The reminder of the paper is organized as follows. Section \ref{2} focuses on deriving the functional linked to contour dynamics equations when dealing with stationary vortex patches for \eqref{1-1}. Furthermore, we introduce the function spaces that will be utilized in subsequent sections. In Section \ref{section3}, we expand the functional and demonstrate that it possesses $C^1$ regularity.  Section \ref{section4} is dedicated to analyze the linearization of the functional introduced in Section \ref{2}. Moreover, we prove that the linearized operator is in fact an isomorphism. In Section \ref{section5},  we investigate the  existence of stationary vortex patches for \eqref{1-1} when $\gamma\in(1,2)$.

    Throughout this paper, we use the symbol $C$ to represent constants, which may vary from one line to another. Additionally, we use the notation $\mathcal{R}_{i,j}$ to denote a general bounded and $C^1$ continuous function, which may also vary from line to line. Furthermore, we will adopt the following notation to represent the mean value of the integral of $g$ over the unit circle
\begin{equation*}
    \int\!\!\!\!\!\!\!\!\!\; {}-{} g(\tau)d\tau:=\frac{1}{2\pi}\int_0^{2\pi}g(\tau)d\tau.
\end{equation*}

\section{Formulation and functional setting}\label{2}

In this section, we will focus on the formulation of the vortex patch motion to the gSQG and construction of the functional that is connected to the contour dynamics equations for stationary patches. We will carefully choose the positions of each $x_{\varepsilon,i}$ to ensure that the linearization operator is invertible within certain functional spaces. In this
setting, the solution takes at least for a short time the form $\omega(t) = \chi_{\Gamma_i}$ , where $\Gamma_i \subset \Omega$
is smooth and will be chosen close to the unit disc $\mathbb{D}$. To set the stage for our approach, we provide some context for the problem at hand. Consider a small constant $0<\rho_0<\frac{1}{4} \min _{i \neq j}\left\{\left|x_{0, i}-x_{0, j}\right|\right\}$.

 Given that our results primarily focus on patches located near a unit disk $\mathbb{D}$, we can make the assumption that the boundary of each patch, denoted as $\partial \Gamma_i$, can be parameterized by a $2\pi$ periodic curve as follows
\begin{equation*}
    z_i(\beta)=x_i+\varepsilon R_i(\beta)(\cos \beta, \sin \beta)
\end{equation*}
with $R_i(\beta)=\rho_i+\varepsilon^{1+\gamma} g_i(\beta), \beta \in[0,2 \pi)$ and $x_i \in B_{\rho_0}\left(x_{0, i}\right)$.

We now seek stationary patch solutions of \eqref{1-1} which satisfy the initial data
$$\omega_{\varepsilon}=\sum_{i=1}^m \frac{1}{\varepsilon^2} \chi_{\Gamma_i},$$ we notice that the total flux of each patch is approximately $\pi \rho_i^2$, with a small correction denoted as $o_{\varepsilon}(1)$. To make sure that the total flux of each patch matches a specific value $\kappa_i$, we need to adjust the values of $\rho_i$ near $\sqrt{\frac{\kappa_i}{\pi}}$. The equations governing this stationary patch solution $\omega_{\varepsilon}=\sum_{i=1}^m \frac{1}{\varepsilon^2} \chi_{\Gamma_i}$ can be written as
\begin{equation*}
    \boldsymbol{v}\left(z_i(\beta)\right) \cdot \boldsymbol{n}\left(z_i(\beta)\right)=0, \quad \forall i=1, \ldots, m.
\end{equation*}
In this context, $\boldsymbol{n}$ denotes the normal vector to the boundary, which is orthogonal to $z_i^{\prime}(\beta)$. The velocity can be obtained from the vorticity by applying the Biot-Savart law and the Green-Stokes formula. More precisely,
\begin{equation*}
    \begin{aligned}
        \boldsymbol{v}\left(z_i(\beta)\right)= & \frac{C_\gamma}{2 \pi \varepsilon^2} \int_0^{2 \pi}   \frac{z_i^{\prime}(\eta)-z_i^{\prime}(\beta)}{\left|z_i(\beta)-z_i(\eta)\right|^\gamma }d   \eta+\sum_{j \neq i}\frac{C_\gamma}{2 \pi \varepsilon^2} \int_0^{2 \pi} \frac{z_j^{\prime}(\eta)} {\left|z_i(\beta)- z_j(\eta)\right|^\gamma} d \eta \\
        & -\sum_{j=1}^m \int_0^{2 \pi} \nabla_x K_\gamma^0\left(z_i(\beta), x_j+\varepsilon \vartheta\cos \eta, \sin \eta)\right)\vartheta d\vartheta d\eta .
    \end{aligned}
\end{equation*}
We define $I^m:=\prod_{i=1}^m\left(\sqrt{\frac{\kappa_i}{\pi}}-\zeta, \sqrt{\frac{\kappa_i}{\pi}}+\zeta\right)$, where $\zeta>0$ is a small constant. For the sake of convenience, we will adopt the following notations
\begin{equation*}
	\begin{split}
		&A(\beta, \eta):=4\sin^2\left(\frac{\beta-\eta}{2}\right),\,\,\,A_{ij}=|x_{i}-x_{j}|^2,\\
		&{\resizebox{.95\hsize}{!}{$B(r,f,\beta, \eta):=4r(g(\beta)+g(\eta))\sin^2\left(\frac{\beta-\eta}{2}\right)+\varepsilon|\varepsilon|^{\gamma}\left((g(\beta)-g(\eta))^2+4 g(\beta)g(\eta)\sin^2\left(\frac{\beta-\eta}{2}\right)\right),$}}\\
		&{\resizebox{.95\hsize}{!}{$B_{ij}(\beta, \eta):=2(x_{i}-x_{j})\cdot(\rho_i(\cos\beta, \sin\beta)-\rho_j(\cos\eta, \sin\eta))+2\varepsilon|\varepsilon|^{\gamma}(x_{i}-x_{j})\cdot\left(g_i(\beta)(\cos\beta, \sin\beta)\right.$}}\\
		&\quad\,\left.-g_{j}(\eta)(\cos\eta, \sin\eta)\right)+\varepsilon((\rho_i+\varepsilon|\varepsilon|^{\gamma} g_i(\beta))(\cos\beta, \sin\beta)-(\rho_j+\varepsilon|\varepsilon|^{\gamma} g_j(\eta))(\cos\eta, \sin\eta))^2.
	\end{split}
\end{equation*}
For $\boldsymbol{\rho}\in I^m$, $\boldsymbol{x}\in \Omega^m$ and $\boldsymbol{g}=\left(g_1, \ldots, g_m\right)$. The task of locating a stationary global vortex patch solution for the generalized quasi-geostrophic equation \eqref{1-1} is equivalent to identifying a root of the functional  $G_i(\varepsilon, \boldsymbol{\rho}, \boldsymbol{x}, \boldsymbol{g})$, where
\begin{equation}\label{funct}
    G_i(\varepsilon, \boldsymbol{\rho}, \boldsymbol{x}, \boldsymbol{g})=G_{i, 1}\left(\varepsilon, \rho_i, g_i\right)+G_{i, 2}(\varepsilon, \boldsymbol{\rho}, \boldsymbol{x}, \boldsymbol{g})+G_{i, 3}(\varepsilon, \boldsymbol{\rho}, \boldsymbol{x}, \boldsymbol{g})=0, \quad \forall i=1, \ldots, m,
\end{equation}
and the terms $G_{i, j}\left(\varepsilon, \rho_i, g_i\right)$  for $j=1,2,3$ are given by
\begin{equation*}
    \begin{aligned}
        & G_{i, 1}\left(\varepsilon, \rho_i, g_i\right)=\frac{C_\gamma}{2 \pi \varepsilon|\varepsilon|^{\gamma} R_i(\beta)}  \times \\
        & \int_0^{2 \pi} \frac{\left(\left(R_i(\beta) R_i(\eta)+R_i^{\prime}(\beta) R_i^{\prime}(\eta)\right)\sin(\beta-\eta)+\left(R_i(\beta) R_i^{\prime}(\eta)-R_i^{\prime}(\beta) R_i(\eta)\right) \cos(\beta-\eta)\right) d \eta}{\left(\rho_i^2 A(\beta, \eta)+\varepsilon|\varepsilon|^\gamma B\left(\rho_i,  g_i, \beta,\eta\right)\right)^{\frac{\gamma}{2}}},
    \end{aligned}
\end{equation*}
\begin{equation*}
    \begin{aligned}
        & G_{i, 2}(\varepsilon, \boldsymbol{\rho}, \boldsymbol{x}, \boldsymbol{g}) =\frac{C_\gamma}{2 \pi \varepsilon R_i(\beta)} \sum_{j \neq i} \\
        & \int_0^{2 \pi} \frac{\left(\left(R_i(\beta) \rho_j(\eta)+R_i^{\prime}(\beta) \rho_j^{\prime}(\eta)\right)\sin(\beta-\eta)+\left(R_i(\beta) \rho_j^{\prime}(\eta)-R_i^{\prime}(\beta) \rho_j(\eta)\right) \cos(\beta-\eta)\right) d \eta}{\left(A_{i, j}+\varepsilon B_{i, j}(\beta, \eta)\right)^{\frac{\gamma}{2}}},
    \end{aligned}
\end{equation*}
and
\begin{equation*}
    \begin{aligned}
        & G_{i,3}(\varepsilon, \boldsymbol{\rho}, \boldsymbol{x}, \boldsymbol{g}) \\
        & =-\frac{1}{\varepsilon R_i(\beta)} \sum_{j=1}^m \int_0^{2 \pi} \int_0^{R_j(\eta)} z_i^{\prime}(\beta)\cdot \nabla_x K_\gamma^0\left(z_i(\beta), x_j+\varepsilon \vartheta\cos \eta, \sin \eta)\right)\vartheta d\vartheta d\eta.
    \end{aligned}
\end{equation*}
For $2 \pi$-periodic function $g_i$  in $X^{k+\gamma-1}$, we have
\begin{equation*}
    \begin{split}
        \int_{0}^{2\pi}g_{i}(\beta)\, d\beta =& \int_{0}^{2\pi}\sum_{j=2}^{\infty}(a_{j}\cos{(j\beta)}+b_{j}\sin{(j\beta)})\,d\beta\\
        =& \sum_{j=2}^{\infty}\Bigl(a_{j}\int_{0}^{2\pi}\cos{(j\beta)}\,d\beta + b_{j}\int_{0}^{2\pi}\sin{(j\beta)}\,d\beta\Bigr)=0
    \end{split}
\end{equation*}
and this implies that the requirement for the total flux in Theorem \ref{thm1} (iii) can be expressed equivalently as follows
\begin{equation*}
    \kappa_i=\frac{\left|\Gamma_i\right|}{\varepsilon^2}=\frac{1}{2} \int_0^{2 \pi}\left(\rho_i+\varepsilon|\varepsilon|^{\gamma} g_i(\beta)\right)^2 d \beta=\pi \rho_i^2+\frac{\varepsilon^{2+2\gamma}}{2} \int_0^{2 \pi} g_i^2(\beta) d \beta,
\end{equation*}
implying
\begin{equation}\label{Gi_eq}
    F_i\left(\varepsilon, \rho_i, g_i\right):=\pi \rho_i^2+\frac{\varepsilon^{2+2\gamma}}{2} \int_0^{2 \pi} g_i^2(\beta) d \beta-\kappa_i=0.
\end{equation}

In the preceding decomposition, we identify two distinct terms: the first terms, $G_{i,1}$ and $G_{i,2}$ corresponds to the functional observed in flat space  $\mathbb{R}^2$, delineating the induced patch effect, as described in \cite{Cas1}. The last term, $G_{i,3}$ characterizes the impact of a rigid boundary on the patch.

\subsection{Functional spaces} \label{sub-2.1}
In order to apply the implicit function theorem, we need
first to fix the function spaces and check the regularity of the functional $G_i$ introduced in \eqref{funct} with respect to these spaces. Hence, for $k\geq 3$, in our proofs, we will often make use of the function spaces defined as follows, with their norms naturally defined as norms of product spaces
\begin{equation*}
    \begin{aligned}
        & X^{k+\gamma-1 } \\
        = & \left\{g \in H^k\colon g(\beta)=\sum_{j=2}^{\infty}\left(a_j \cos (j \beta)+b_j \sin (j       \beta)\right),\left\|\int_0^{2 \pi} \frac{\partial^k g(\beta-\eta)-\partial^k g(\beta)}{\left|\sin  \left(\frac{\eta}{2}\right)\right|^{\gamma}} d \eta\right\|_{L^2}<\infty\right\},
    \end{aligned}
\end{equation*}
with the norm
\begin{equation*}
    \|g\|_{X^{k+\gamma-1 }}=\|g\|_{H^k}+\left\|\int_0^{2 \pi} \frac{\partial^k g(\beta-\eta)-\partial^k g(\beta)}{\left|\sin \left(\frac{\eta}{2}\right)\right|^{\gamma}} d \eta\right\|_{L^2},
\end{equation*}
and
\begin{equation*}
    Y^k=\left\{g \in H^k, g(\beta)=\sum_{j=1}^{\infty}\left(a_j \sin (j \beta)+b_j \cos (j \beta)\right)\right\}
\end{equation*}
and
\begin{equation*}
    Y_0^k=Y^k / \operatorname{span}\{\sin (\beta), \cos (\beta)\}=\left\{g \in H^k, g(\beta)=\sum_{j=2}^{\infty}\left(a_j \sin (j \beta)+b_j \cos (j \beta)\right)\right\},
\end{equation*}
with the standard $H^k$-norm.

\section{Extension and regularity of functionals}\label{section3}

In order to apply the implicit function theorem at $\varepsilon=0$, it is necessary to extend the functions $G_{i,1}$ and $G_{i,2}$, which were defined in Section \ref{2}, to the domain where $\varepsilon\leq 0$, ensuring that they maintain $C^1$ regularity.

Denote $V_{\mu}$ as the open neighborhood of zero in $(X^{k}_{k+\gamma-1})^m$ defined by
\begin{equation*}
	V_\mu:=\left\{\boldsymbol{g}=(g_1,\cdots,g_m)\in (X^{k}_{k+\gamma-1})^m: \ \sum_{j=1}^m\|g\|_{X^{k+\gamma-1}}<\mu\right\}
\end{equation*}
with $0<\mu<1$ and $k\ge 3$.  Let also  $B_{\rho_0}\left(\boldsymbol{x}_0\right)$ be the ball centered at $\boldsymbol{x}_0$ in $\Omega^m$ with a small radius $0<\rho_0<\frac{1}{4} \min _{i \neq j}\left\{\left|x_{0, i}-x_{0, j}\right|\right\}$. The initial step involves verifying the continuity of $G_i$.

\begin{lemma}\label{p3-1}
	There exists $\varepsilon_0>0$ such that for each $i$, the functional $G_i$ can be extended from $\left(-\varepsilon_0, \varepsilon_0\right) \times I^m \times B_{\rho_0}\left(x_0\right) \times V_\mu$ to $\left(Y^{k-1}\right)^m$ as  continuous functional.
\end{lemma}
\begin{proof}
The proof of continuity  for the term $G_{i, 1}$  that we shall present here is more or less classical and is analogous to \cite{cao,Cas1}, and we shall provide a complete proof for the self-containing of the paper. Consequently, we can decompose the functional
 $G_{i, 1}$ as follows
\begin{equation}\label{2-2}
	\begin{split}
		G_{i, 1}=&\frac{C_\gamma}{\varepsilon|\varepsilon|^{\gamma}}\int\!\!\!\!\!\!\!\!\!\; {}-{} \frac{(\rho_i+\varepsilon|\varepsilon|^\gamma g_i(\eta))\sin(\beta-\eta)d\eta}{\left(\rho_i^2 A(\beta, \eta)+\varepsilon|\varepsilon|^\gamma B\left(\rho_i, g_i, \beta, \eta\right)\right)^{\frac{\gamma}{2}}}
  +C_\gamma\int\!\!\!\!\!\!\!\!\!\; {}-{} \frac{(g'_i(\eta)-g'_i(\beta))\cos(\beta-\eta)d\eta}{\left(\rho_i^2 A(\beta, \eta)+\varepsilon|\varepsilon|^\gamma B\left(\rho_i, g_i, \beta, \eta\right)\right)^{\frac{\gamma}{2}}}\\		
		&
  +\frac{C_\gamma \varepsilon |\varepsilon|^\gamma g'_i(\beta)}{\rho_i+\varepsilon|\varepsilon|^\gamma g_i(\beta)}\int\!\!\!\!\!\!\!\!\!\; {}-{} \frac{(g_i(\beta)-g_i(\eta))\cos(\beta-\eta)d\eta}{\left(\rho_i^2 A(\beta, \eta)+\varepsilon|\varepsilon|^\gamma B\left(\rho_i, g_i, \beta, \eta\right)\right)^{\frac{\gamma}{2}}}\\	
		&
  +\frac{C_\gamma\varepsilon|\varepsilon|^\gamma}{\rho_i+\varepsilon|\varepsilon|^\gamma g_i(\beta)}\int\!\!\!\!\!\!\!\!\!\; {}-{} \frac{g'_i(\beta)g'_i(\eta)\sin(\beta-\eta)d\eta}{\left(\rho_i^2 A(\beta, \eta)+\varepsilon|\varepsilon|^\gamma B\left(\rho_i, g_i, \beta, \eta\right)\right)^{\frac{\gamma}{2}}}\\	
		=&G_{i11}+G_{i12}+G_{i13}+G_{i14}.
	\end{split}
\end{equation}
Specifically, the most singular component in \eqref{2-2}  is given by $G_{i11}$. To prove the continuity  to this term we proceed as follows
\begin{equation*}
    G_{i11}:=\frac{C_\gamma}{\varepsilon|\varepsilon|^{\gamma}}\int\!\!\!\!\!\!\!\!\!\; {}-{} \frac{(\rho_i+\varepsilon|\varepsilon|^{\gamma} g_i(\eta))\sin(\beta-\eta)}{\left(\rho_i^2 A(\beta, \eta)+\varepsilon|\varepsilon|^{\gamma} B\left(\rho_i, g_i, \beta, \eta\right)\right)^{\frac{\gamma}{2}}} d\eta,
\end{equation*}
since $\rho_i(\beta)=\rho_i+\varepsilon|\varepsilon|^{\gamma} g_i(\beta), \beta \in[0,2 \pi)$ and $x_i \in B_{\rho_0}\left(x_{0, i}\right)$, the possible singularity caused by $\varepsilon=0$ may occur only when we take zero-th order derivative of $G_{i,11}$. Throughout the proof, we will frequently use the following Taylor's formula
\begin{equation}\label{taylor}
	\frac{1}{(A+B)^s}=\frac{1}{A^s}-s\int_0^1\frac{B}{(A+tB)^{1+s}}dt.
\end{equation}
Then, we decompose the kernel into two parts
\begin{equation}\label{2-6}
	\begin{split}
		G_{i11}&=\frac{C_{\gamma}}{\varepsilon|\varepsilon|^{\gamma}}\int\!\!\!\!\!\!\!\!\!\; {}-{}\frac{\rho_i\sin{(\beta-\eta)}d\eta}{\left(\rho_i^{2} A(\beta,\eta)+\varepsilon|\varepsilon|^{\gamma}B\left(\rho_i,G_{i},\beta,\eta\right)\right)^{\frac{\gamma}{2}}}+C_\gamma\int\!\!\!\!\!\!\!\!\!\;{}-{} \frac{g_{i}(\eta)\sin{(\beta-\eta)}d\eta}{\left(\rho_i^{2} A(\beta,\eta)+\varepsilon|\varepsilon|^{\gamma} B\left(\rho_i, G_{i}, \beta, \eta\right)\right)^{\frac{\gamma}{2}}}\\
		&=\frac{C_\gamma \rho_i}{\varepsilon|\varepsilon|^{\gamma}\rho_i^\gamma}\int\!\!\!\!\!\!\!\!\!\; {}-{} \frac{\sin(\beta-\eta)d\eta}{ A^{\frac{\gamma}{2}}}-\frac{\gamma C_\gamma  \rho_i}{2}\int\!\!\!\!\!\!\!\!\!\; {}-{} \int_0^1 \frac{B\sin(\beta-\eta)dt d\eta}{\left(\rho_i^2 A(\beta, \eta)+t\varepsilon|\varepsilon|^\gamma B\left(\rho_i, g_i, \beta, \eta\right)\right)^{\frac{\gamma+2}{2}}}\\
        &\qquad \qquad+ C_{\gamma}\int\!\!\!\!\!\!\!\!\!\; {}-{} \frac{g_i(\eta)\sin{(\beta-\eta)}d\eta}{\left(\rho_i^{2} A(\beta,\eta)+\varepsilon|\varepsilon|^{\gamma} B\left(\rho_i,G_{i},\beta,\eta\right)\right)^{\frac{\gamma}{2}}}\\
		&=-\frac{\gamma C_\gamma \rho_i}{2}\int\!\!\!\!\!\!\!\!\!\; {}-{} \int_0^1 \frac{B\sin(\beta-\eta)dt d\eta}{(\rho_i^{2}A)^{\frac{\gamma}{2}+1}}\\
  &
\ \ \ \ + \frac{C_\gamma\varepsilon|\varepsilon|^\gamma \gamma(\gamma+2)\rho_i}{4}\int\!\!\!\!\!\!\!\!\!\; {}-{}\int_0^1 \int_0^1 \frac{t B^{2}\sin(\beta-\eta)d\tau dt d\eta}{\left(\rho_i^2 A(\beta, \eta)+t\tau\varepsilon|\varepsilon|^\gamma B\left(\rho_i, g_i, \beta, \eta\right)\right)^{\frac{\gamma+4}{2}}}\\
		&\ \ \ \ +\frac{C_\gamma}{\rho_i^\gamma}\int\!\!\!\!\!\!\!\!\!\; {}-{} \frac{g_i(\eta)\sin(\beta-\eta)d\eta}{A^{\frac{\gamma}{2}}}-\frac{C_\gamma \gamma\varepsilon|\varepsilon|^\gamma}{2}\int\!\!\!\!\!\!\!\!\!\; {}-{} \int_0^1 \frac{Bg_i(\eta)\sin(\beta-\eta)dt d\eta}{\left(\rho_i^2 A(\beta, \eta)+\varepsilon|\varepsilon|^\gamma B\left(\rho_i, g_i, \beta, \eta\right)\right)^{\frac{\gamma+2}{2}}}\\
        &= -\frac{\gamma C_\gamma \rho_i}{2\rho_i^{\gamma+2}}\int\!\!\!\!\!\!\!\!\!\; {}-{}  \frac{\rho_ig_i(\eta)A\sin(\beta-\eta)d\eta}{A^{\frac{\gamma+2}{2}}}+\frac{C_\gamma}{\rho_i^\gamma}\int\!\!\!\!\!\!\!\!\!\; {}-{} \frac{g_i(\eta)\sin(\beta-\eta)d\eta}{A^{\frac{\gamma}{2}}}+\varepsilon|\varepsilon|^\gamma\mathcal{R}_{i11}\\
		&=\frac{C_\gamma}{\rho_i^\gamma}\left(1-\frac{\gamma}{2}\right)\int\!\!\!\!\!\!\!\!\!\; {}-{} \frac{g_i(\eta)\sin(\beta-\eta)d\eta}{\left|\sin(\frac{\beta-\eta}{2})\right|^\gamma}+\varepsilon|\varepsilon|^\gamma\mathcal{R}_{i11},
	\end{split}
\end{equation}
where we used the fact that $B(\rho_i,g_i,\beta, \eta)=\rho_iA(g_i(\eta)+(g_i(\beta)))+\varepsilon R_i$
with $\mathcal{R}_{i11}$ and $R_i$ not singular with respect to $\varepsilon$.

Next, we proceed to differentiate $G_i$ with respect to $\beta$ up to $\partial^{k-1}$ times. Our initial focus will be on the most singular term, namely, $G_{i12}$.
\begin{equation*}
    \begin{split}
    	&\partial^{k-1}G_{i12}=C_\gamma\int\!\!\!\!\!\!\!\!\!\; {}-{}\frac{(\partial^kg_i(\eta)-\partial^kg_i(\beta))\cos(\beta-\eta)d\eta}{\left( |\varepsilon|^{2+2\gamma}\left(g_i(\beta)-g_i(\eta)\right)^2+4(\rho_i+\varepsilon|\varepsilon|^\gamma g_i(\beta)(\rho_i+\varepsilon|\varepsilon|^\gamma g_i(\eta))\sin^2\left(\frac{\beta-\eta}{2}\right)\right)^{\frac{\gamma}{2}}}\\	
    	&
        \ \  -\frac{ \gamma C_\gamma |\varepsilon|^{1+\gamma}}{2}\int\!\!\!\!\!\!\!\!\!\; {}-{}\frac{\cos(\beta-\eta)}{\left( |\varepsilon|^{2+2\gamma}\left(g_i(\beta)-g_i(\eta)\right)^2+4(\rho_i+\varepsilon|\varepsilon|^\gamma g_i(\beta))(\rho_i+\varepsilon|\varepsilon|^\gamma g_i(\eta))\sin^2\left(\frac{\beta-\eta}{2}\right)\right)^{\frac{\gamma+2}{2}}}\\
    	&
        \ \  \times\left(2|\varepsilon|^{1+\gamma}(g_i(\beta)-g_i(\eta))(g'_i(\beta)-g'_i(\eta))+4(R_i(\beta)g'_i(\eta)+g'_i(\beta)R_i(\eta))\sin^{2}\left(\frac{\beta-\eta}{2}\right)\right)\\
    	&\ \  \times(\partial^{k-1}g_i(\eta)-\partial^{k-1}g_i(\beta))d\eta +l.o.t,
    \end{split}
\end{equation*}
where ``l.o.t'' denotes the lower order terms. Since $g_{i}(\beta)\in X^{k+\gamma-1}$ and $k\ge3$, we have $\|\partial^{N}g\|_{L^\infty}\leq C \|g\|_{X^{k+\gamma-1}}<\infty$ for $N=0,1,2$. By H\"older's inequality and mean value theorem, we can conclude that
\begin{equation*}
    \begin{split}
        \left\|\partial^{k-1}G_{i12}\right\|_{L^2}&\leq C\left\|\int\!\!\!\!\!\!\!\!\!\;{}-{} \frac{\partial^{k}g_{i}(\beta)-\partial^{k}g_{i}(\eta)}{|\sin(\frac{\beta-\eta}{2})|^{\gamma}}d\eta\right\|_{L^2}
        +C\left\|\int\!\!\!\!\!\!\!\!\!\; {}-{}\frac{\partial^{k- 1}g_i(\beta)-\partial^{k-1}g_i(\eta)}{|\sin(\frac{\beta-\eta}{2})|^{\gamma}}d\eta\right\|_{L^2}\\
        &\leq C \|g\|_{X^{k+\gamma-1}}+C\|g\|_{X^{k+\gamma-2}}<\infty.
    \end{split}
\end{equation*}
It's worth noting that obtaining bounds for $\left\|\partial^{k-1}G_{i13}\right\|_{L^2}$, $\left\|\partial^{k-1}G_{i12}\right\|_{L^2}$, $\left\|\partial^{k-1}G_{i14}\right\|_{L^2}$ is comparatively simpler than acquiring the bound for $\left\|\partial^{k-1}G_{i11}\right\|_{L^2}$. Consequently, we can conclude that the range of $F_{i,1}$ is within $Y^{k-1}$.

To establish the continuity of $G_{i,1}$, we will focus our attention on demonstrating continuity for the most singular term, namely, $G_{i12}$. Specifically, for any $g_{i1},g_{i2}\in X^{k+\gamma-1}$, we aim to show that the following inequality holds: $\norm{G_{i12}(\varepsilon, g_{i1})-G_{i12}(\varepsilon, g_{i2})}_{Y^{k-1}}\le C\norm{g_{i1}-g_{i2}}_{X^{k+\gamma-1}}$. For a more general function $g$, we introduce the following definition
\begin{equation*}
    \Delta g=g(\beta)-g(\eta),\ \mbox{where}\ g=g(\beta),\ \mbox{and }\,
    \tilde{g}=g(\eta)
\end{equation*}%
and
\begin{equation*}
    D_{\gamma }(g_i)=\varepsilon ^{2+2\gamma }\Delta g_i^{2}+4(\rho_i+\varepsilon |\varepsilon|^{\gamma}g_i)(\rho_i+\varepsilon |\varepsilon|^{\gamma }\tilde{g}_i)\sin^{2}\left(\frac{\beta-\eta}{2}\right) .
\end{equation*}
Then for $g_{i1},g_{i2}\in V_{\mu}$, $i=1,2,\ldots,m$, we have
\begin{equation*}
    \begin{split}
        G_{i12}(\varepsilon ,g_{i1})& -G_{i12}(\varepsilon ,g_{i2})=C_{\gamma}\int
        \!\!\!\!\!\!\!\!\!\;{}-{}\frac{(\Delta g_{i1}^{\prime}-\Delta
        g_{i2}^{\prime })\cos (\beta-\eta)d\eta}{D_{\gamma}(g_{i1})^{\frac{1}{2}}} \\
        & +\left( C_{\gamma}\int \!\!\!\!\!\!\!\!\!\;{}-{}\frac{\Delta
        g_{i2}^{\prime }\cos (\beta-\eta)d\eta}{D_{\gamma}(g_{i1})^{\frac{1}{2}}}-C_{\gamma}\int
        \!\!\!\!\!\!\!\!\!\;{}-{}\frac{\Delta g_{i2}^{\prime }\cos(\beta-\eta)d\eta}{D_{\gamma}(g_{i2})^{\frac{1}{2}}}\right)\\
        & =I_{1}+I_{2}.
    \end{split}
\end{equation*}
We can easily observe that $I_{1}$ is bounded by $\Vert I_{1}\Vert_{Y^{k-1}}\leq C\Vert g_{i1}-g_{i2}\Vert_{X^{k+\gamma-1 }}$. To estimate $I_{2}$, we can employ the mean value theorem in the following manner:
\begin{equation}
    \begin{split}
        & \frac{1}{D_{\gamma }(g_{i1})^{\frac{\gamma }{2}}}-\frac{1}{D_{\gamma
        }(g_{i2})^{\frac{\gamma }{2}}}=\frac{D_{\gamma
        }(g_{i2})^{\frac{\gamma }{2}}-D_{\gamma }(g_{i1})^{\frac{\gamma }{2}}}{D_{\gamma }(g_{i1})^{\frac{\gamma }  {2}%
        }D_{\gamma }(g_{i2})^{\frac{\gamma }{2}}} \\
        &=\frac{\gamma }{2}\frac{D_{\gamma
        }(g_{i2})-D_{\gamma }(g_{i1})}{D_{\gamma }(\delta _{\beta,\eta}g_{i1}+(1-\delta
        _{\beta,\eta})g_{i2})^{1-\frac{\gamma }{2}}D_{\gamma }(g_{i1})^{\frac{\gamma }{2}%
        }D_{\gamma }(g_{i2})^{\frac{\gamma }{2}}} \\
        & {\resizebox{.98\hsize}{!}{$=\frac{\gamma}{2}\frac{|%
        \varepsilon|^{2+2\gamma}(\Delta g_{i2}^2-\Delta
        g_{i1}^2)+4\varepsilon|\varepsilon|^\gamma
        ((g_{i2}-g_{i1})(\rho_i+\varepsilon|\varepsilon|^\gamma  \tilde
        g_{i2})+(\tilde g_{i2}-\tilde g_{i1})(\rho_i+\varepsilon|\varepsilon|^\gamma
        g_{i1}))\sin^2(\frac{\beta-\eta}{2})}{D_\gamma(\delta_{\beta,\eta}g_{i1}+(1-%
        \delta_{\beta,\eta})g_{i2})^{1-\frac{\gamma}{2}}D_\gamma(g_{i1})^\frac{\gamma}{2}D_%
        \gamma(g_{i2})^\frac{\gamma}{2}}$}},
    \end{split}\label{2-7}
\end{equation}%
for some $\delta _{\beta,\eta}\in (0,1)$. Now, by combining \eqref{2-7} and the fact  $D_{\gamma}(g_i)\sim 4 \rho_i^2\sin
^{2}(\frac{\beta-\eta}{2})\sim \rho_i^2|\beta-\eta|^{2}$ as $|\beta-\eta|\rightarrow 0$ one gets
\begin{equation*}
    \begin{split}
        & \partial^{k-1}I_2\sim \frac{C}{\rho_i^\gamma}\displaystyle\int\!\!\!\!\!\!\!\!\!\;
        {}-{}\frac{\partial^{k-1}g_{i2}(\beta)-\partial^{k-1}g_{i2}(\eta)}{|\sin(%
        \frac{\beta-\eta}{2})|^\gamma}\\
        &
        \qquad\qquad \qquad\times \left(\frac{|\varepsilon|^{2+2\gamma}(\Delta
        g_{i2}^2-\Delta
        g_{i1}^2)}{|\beta-\eta|^{1+\gamma}}+4\varepsilon|\varepsilon|^\gamma(g_{i2}-g_{i1}+\tilde
        g_{i2}-\tilde g_{i1})\right)d\eta +l.o.t.
    \end{split}%
\end{equation*}%
From this analysis, we can conclude that $\|I_2\|_{Y^{k-1}}\le C\|g_{i1}-g_{i2}\|_{X^{k+\gamma-1}}$. Thus, we have demonstrated that $G_{i,1}(\varepsilon, g): \left(-\frac{1}{2}, \frac{1}{2}\right)\times V_{\mu} \rightarrow Y^{k-1}$ is continuous. In equation \eqref{2-6}, we have already assembled the primary components of $G_{i,1}$.  Now, we will also make use of the Taylor formula \eqref{taylor} on $G_{i12}$ and $G_{i13}$ to obtain
\begin{equation}\label{2-8}
    \begin{split}
    	G_{i,1}&=\frac{C_\gamma}{\rho_i^\gamma}\left(1-\frac{\gamma}{2}\right)\int\!\!\!\!\!\!\!\!\!\; {}-{} \frac{g_i(\eta)\sin(\beta-\eta)d\eta}{|\sin(\frac{\beta-\eta}{2})|^\gamma}-\frac{C_\gamma}{\rho_i^\gamma}\int\!\!\!\!\!\!\!\!\!\; {}-{} \frac{(g'_i(\eta)-g'_i(\beta))\cos(\beta-\eta)d\eta}{|\sin(\frac{\beta-\eta}{2})|^\gamma}+\varepsilon|\varepsilon|\mathcal{R}_{i,1}\\
    	&=\frac{C_\gamma}{\rho_i^\gamma}\left(1-\frac{\gamma}{2}\right)\int\!\!\!\!\!\!\!\!\!\; {}-{} \frac{g_i(\beta-\eta)\sin(\eta)d\eta}{\sin(\frac{\eta}{2})}-\frac{C_\gamma}{\rho_i^\gamma}\int\!\!\!\!\!\!\!\!\!\; {}-{} \frac{(g'_i(\beta)-g'_i(\beta-\eta))\cos(\eta)d\eta}{\sin(\frac{\eta}{2})}+\varepsilon|\varepsilon|\mathcal{R}_{i,1},
    \end{split}
\end{equation}
where $\mathcal{R}_{i,1}(\varepsilon, f): \left(-\frac{1}{2}, \frac{1}{2}\right)\times V_{\mu} \rightarrow Y^{k-1}$ is continuous by previous discussion.

Since $K_\gamma^0(x,y)$ is smooth in $\Omega$, the terms $G_{i,2}$ and $G_{i,3}$ are apparently smooth and belong to $H^k$. It is noteworthy that $z_i^{\prime}(\beta)=$ $\varepsilon R_i(\beta)(-\sin \beta, \cos \beta)+\varepsilon^{2+\gamma} g_i^{\prime}(\beta)(\cos \beta, \sin \beta)$. Consequently, the $\varepsilon$ in the denominator of the initial term in $G_{i, 3}$ gets eliminated. A simple observation reveals that the integral within $G_{i, 2}$ equals to 0 when $\varepsilon=0$. In fact, it can be demonstrated that the integral within $G_{i, 2}$ possesses an expansion of the type $C_1 \varepsilon+O\left(\varepsilon^2\right)$ by applying Taylor formula, where $C_1$ is a constant uninfluenced by $\varepsilon$ (refer to \eqref{333}). More precisely, the majority of  the linear terms in the quotient resulting of the integral in $G_{i, 2}$ contain a factor with an exponent greater than 1 in relation to $\varepsilon$. Then, we can see those terms as are remaining terms.  The single linear term that possesses an exponent of 1 in the $\varepsilon$ factor  is the term  derived from the expansion  of the denominator which is obtained by the Taylor formula. This demonstrates that the $\varepsilon$ in the denominator of the initial term in $G_{i, 2}$ is also eliminated. Consequently, the continuity of the remaining  terms, $G_{i, 2}$ and $G_{i, 3}$, can be easily verified.  In fact, we first treat the term  $G_{i, 2}$ defined in Section \ref{2}. By direct computations, for $\varepsilon$ small and using the Taylor formula \eqref{taylor}, we find

\begin{equation}\label{333}
    \begin{aligned}
        & G_{i, 2}(\varepsilon, \boldsymbol{\rho}, \boldsymbol{x}, \boldsymbol{g})=
\\
        &=\frac{C_\gamma}{\varepsilon}\sum_{j \neq i}\int\!\!\!\!\!\!\!\!\!\; {}-{} \frac{(\rho_j+\varepsilon|\varepsilon|^\gamma g_j(\eta))\sin(\beta-\eta)d\eta}{\left(A_{i, j}+\varepsilon B_{i, j}(\beta, \eta)\right)^{\frac{\gamma}{2}}}+ \frac{C_\gamma\varepsilon|\varepsilon|^{2\gamma}}{\rho_i+\varepsilon|\varepsilon|^\gamma g_i(\beta)}\int\!\!\!\!\!\!\!\!\!\; {}-{} \frac{g'_i(\beta)g'_j(\eta)\sin(\beta-\eta)d\eta}{\left(A_{i, j}+\varepsilon B_{i, j}(\beta, \eta)\right)^{\frac{\gamma}{2}}}\\
		& \quad +C_\gamma |\varepsilon|^\gamma\int\!\!\!\!\!\!\!\!\!\; {}-{} \frac{g'_j(\eta)\cos(\beta-\eta)d\eta}{\left(A_{i, j}+\varepsilon B_{i, j}(\beta, \eta)\right)^{\frac{\gamma}{2}}}-\frac{C_\gamma  |\varepsilon|^\gamma g'_i(\beta)}{\rho_i+\varepsilon|\varepsilon|^{\gamma} g_i(\beta)}\int\!\!\!\!\!\!\!\!\!\; {}-{} \frac{(\rho_j+\varepsilon|\varepsilon|^\gamma g_j(\eta))\cos(\beta-\eta)d\eta}{\left(A_{i, j}+\varepsilon B_{i, j}(\beta, \eta)\right)^{\frac{\gamma}{2}}}\\			
        & =\frac{C_\gamma}{\varepsilon}\sum_{j \neq i}\int\!\!\!\!\!\!\!\!\!\; {}-{} \frac{\rho_j\sin(\beta-\eta)d\eta}{\left(A_{i, j}+\varepsilon B_{i, j}(\beta, \eta)\right)^{\frac{\gamma}{2}}}+\varepsilon \mathcal{R}_{i, 2}\\
        &=\frac{C_\gamma}{\varepsilon}\sum_{j \neq i}\int\!\!\!\!\!\!\!\!\!\; {}-{} \frac{\rho_j\sin(\beta-\eta)d\eta}{\left(A_{i, j}\right)^{\frac{\gamma}{2}}}-\frac{\gamma C_\gamma}{2\varepsilon}\sum_{j \neq i}\int\!\!\!\!\!\!\!\!\!\; {}-{}\int_0^1\frac{\varepsilon \rho_j B_{i,j}\sin(\beta-\eta)d\eta d\tau}{\left(A_{i, j}+\tau\varepsilon B_{i, j}(\beta, \eta)\right)^{\frac{\gamma}{2}+1}}+\varepsilon \mathcal{R}_{i, 2} \\
        &=-{\resizebox{.9\hsize}{!}{$\frac{\gamma C_\gamma}{2\varepsilon}\sum_{j \neq i}  \int\!\!\!\!\!\!\!\!\!\; {}-{}\frac{\varepsilon \rho_j B_{i,j}\sin(\beta-\eta)d\eta }{\left(A_{i, j}\right)^{\frac{\gamma}{2}+1}}-\frac{\gamma (\gamma+2)C_\gamma}{4\varepsilon}\sum_{j \neq i}  \int\!\!\!\!\!\!\!\!\!\; {}-{}\int_0^1\int_0^1\frac{\varepsilon^2 \rho_j \tau B^2_{i,j}\sin(\beta-\eta)d\eta d\tau dt}{\left(A_{i, j}+t\tau\varepsilon B_{i, j}(\beta, \eta)\right)^{\frac{\gamma}{2}+2}}+\varepsilon \mathcal{R}_{i, 2} $}} \\
        &= -\frac{\gamma C_\gamma}{2\varepsilon}\sum_{j \neq i}\int\!\!\!\!\!\!\!\!\!\; {}-{}\frac{\varepsilon \rho_j \left[2(x_{i}-x_{j})\cdot(\rho_i(\cos\beta, \sin\beta)-\rho_j(\cos\eta, \sin\eta))\right]\sin(\beta-\eta)d\eta }{\left(A_{i, j}\right)^{\frac{\gamma}{2}+1}}+\varepsilon \mathcal{R}_{i, 2}\\
        &=\frac{\gamma C_\gamma}{2\pi} \sum_{j \neq i} \rho_j^2 \int_0^{2 \pi} \frac{\sin (\beta-\eta)\left(x_i-x_j\right) \cdot(\cos \eta, \sin \eta)}{\left(A_{i, j}\right)^{\frac{\gamma}{2}+1}} d \eta+\varepsilon \mathcal{R}_{i, 2} \\
        &= \frac{\gamma C_\gamma}{2}\sum_{j \neq i} \rho_j^2 \frac{\left(x_i-x_j\right) \cdot(\sin \beta,-\cos \beta)}{\left|x_i-x_j\right|^{\gamma+2}}+\varepsilon \mathcal{R}_{i,2}.
    \end{aligned}
\end{equation}
Here, $\mathcal{R}_{i,2}$ refers to a generic bounded function that is continuous. We have utilized the identity $\int_0^{2 \pi} \sin (\beta-\eta) d \eta=0$. We now turn to $G_{i,3}$ and note that
\begin{equation*}
    \begin{split}
        G_{i,3}=&\frac{-1}{\varepsilon R_i(\beta)} \sum_{j=1}^{m}\int_{0}^{2\pi}\int_{0}^{R_{j}(\eta)} z'_{i}(\beta)\cdot\nabla_{x}K_\gamma^0(z_{i}(\beta),x_{j}+\varepsilon  \vartheta(\cos{\eta},\sin{\eta})))\vartheta d\vartheta\,d\eta\\
        =&-\sum_{j=1}^{m}\int_{0}^{2\pi}\int_{0}^{R_{j}(\eta)}(-\sin{\beta},\cos{\beta})\cdot\nabla_{x}K_\gamma^0(z_{i}(\beta),x_{j}+\varepsilon \vartheta(\cos{\eta},\sin{\eta})))\vartheta d\vartheta\,d\eta\\
        &-\frac{\varepsilon^{2} g_{i}'(\beta)}{R_i(\beta)}\sum_{j=1}^{m}\int_{0}^{2\pi}\int_{0}^{R_{j}(\eta)}((\cos{\beta},\sin{\beta}))\cdot\nabla_{x}K_\gamma^0(z_{i}(\beta),x_{j}+\varepsilon \vartheta(\cos{\eta},\sin{\eta})))\vartheta d\vartheta\,d\eta.\\
    \end{split}
\end{equation*}
By Mean Value Theorem for Integrals, we have
\begin{equation*}
    \begin{split}
       \nabla_{x}K_\gamma^0(z_{i}(\beta),x_{j}&+\varepsilon \vartheta(\cos{\eta},\sin{\eta})))=\nabla_{x}K_\gamma^0(z_{i}(\beta),x_{j})\\
       &+\left(\int_{0}^{1}\partial_{t}\nabla_{x}K_\gamma^0(z_{i}(\beta),x_{j}+t\varepsilon \vartheta(\cos{\eta},\sin{\eta})))\,dt\right)\cdot\varepsilon \vartheta(\cos{\eta},\sin{\eta}))\\
       =&\nabla_{x}K_\gamma^0(x_{i},x_{j}) + \left(\int_{0}^{1}\partial_{t}\nabla_{x}K_\gamma^0(x_{i}+t\varepsilon \vartheta(\cos{\beta},\sin{\beta})),x_{j})\,dt\right)\cdot\varepsilon \vartheta(\cos{\beta},\sin{\beta}))\\
       &+ \left(\int_{0}^{1}\partial_{t}\nabla_{x}K_\gamma^0(z_{i}(\beta),x_{j}+t\varepsilon \vartheta(\cos{\eta},\sin{\eta})))\,dt\right)\cdot\varepsilon   \vartheta(\cos{\eta},\sin{\eta})).\\
    \end{split}
\end{equation*}
Therefore,
\begin{equation*}
    \begin{split}
        G_{i,3}=&\sum_{j=1}^{m}\int_{0}^{2\pi}\int_{0}^{R_{j}(\eta)}(\sin{\beta},-\cos{\beta})\cdot\nabla_{x}K_\gamma^0(x_{i},x_{j})\vartheta\,d\vartheta\,d\eta
        +\varepsilon\mathcal{R}_{i,3}\\
        =&\frac{1}{2}\sum_{j=1}^{m}(\sin{\beta},-\cos{\beta})\cdot\nabla_{x}K_\gamma^0(x_{i},x_{j})\int_{0}^{2\pi}R_{j}^{2}(\eta)\,d\eta+\varepsilon\mathcal{R}_{i,3}\\
        =&\frac{1}{2}\sum_{j=1}^{m}(\sin{\beta},-\cos{\beta})\cdot\nabla_{x}K_\gamma^0(x_{i},x_{j})\int_{0}^{2\pi}(\rho_{j}+\varepsilon^{1+\gamma}g_{j}(\eta))^{2}\,d\eta +\varepsilon\mathcal{R}_{i,3}\\
        =&\frac{1}{2}\sum_{j=1}^{m}(\sin{\beta},-\cos{\beta})\cdot\nabla_{x}K_\gamma^0(x_{i},x_{j})\int_{0}^{2\pi}(\rho_{j}^{2}+2\rho_{j}\varepsilon^{1+\gamma}g_{j}(\eta)+\varepsilon^{2+2\gamma}g_{j}^{2}(\eta))\,d\eta +\varepsilon\mathcal{R}_{i,3}\\
        =&\sum_{j=1}^{m}(\sin{\beta},-\cos{\beta})\cdot\nabla_{x}K_\gamma^0(x_{i},x_{j}) \pi \rho_{j}^{2} + \varepsilon\mathcal{R}_{i,3},
    \end{split}
\end{equation*}
where  $\mathcal{R}_{i,3}$ represents a bounded and continuous function that may vary from line to line. As a result, we have demonstrated
\begin{equation}\label{3-6}
    G_{i, 3}(\varepsilon, \boldsymbol{\rho}, \boldsymbol{x}, \boldsymbol{g})=\sum_{j=1}^m \pi \rho_j^2(\sin \beta,-\cos \beta) \cdot \nabla_x K_\gamma^0\left(x_i, x_j\right)+\varepsilon \mathcal{R}_{i, 3} .
\end{equation}
Thus, we conclude the proof of the continuity of the functional $g_i$.
\end{proof}

By combining  \eqref{2-8}, \eqref{333} and \eqref{3-6} one gets the following expressions
\begin{equation}\label{3-7}
    \begin{cases}
        G_{i,1}&=\dfrac{C_\gamma}{\rho_i^\gamma}\left(1-\frac{\gamma}{2}\right)\displaystyle\int\!\!\!\!\!\!\!\!\!\; {}-{}    \frac{g_i(\beta-\eta)\sin(\eta)d\eta}{\sin(\frac{\eta}{2})}-\frac{C_\gamma}    {\rho_i^\gamma}\displaystyle\int\!\!\!\!\!\!\!\!\!\; {}-{} \frac{(g'_i(\beta)-g'_i(\beta-\eta))\cos(\eta)d\eta}{\sin(\frac{\eta}{2})}+\varepsilon|\varepsilon|^\gamma\mathcal{R}_{i,1}, \\
        G_{i,2}&= \mathlarger{\sum}_{j \neq i} \kappa_j  \dfrac{\gamma C_\gamma(x_i-x_j)}{2\pi|x_i-      x_j|^{\gamma+2}}\cdot (\sin\beta, -\cos\beta)+\varepsilon \mathcal{R}_{i, 2},\\
        G_{i,3}&=\mathlarger{\sum}_{j=1}^m \kappa_j  \nabla_x K_\gamma^0\left(x_i, x_j\right) \cdot(\sin \beta,-\cos \beta)+\varepsilon \mathcal{R}_{i, 3} ,%
    \end{cases}
\end{equation}%
where $\mathcal{R}_{i,1}$, $\mathcal{R}_{i,2}$ and $\mathcal{R}_{i,3}$ are bounded and smooth.\medskip

The subsequent proposition focuses on the $C^1$ smoothness of the functional $G_i.$
\begin{lemma}\label{lem2-3}
    For each $(\varepsilon, \boldsymbol{\rho}, \boldsymbol{x}, \boldsymbol{g})\in \left(-\varepsilon_0,    \varepsilon_0\right)\times I^m \times B_{\boldsymbol{\rho}_0}\left(\boldsymbol{x}_0\right) \times V_\mu$, the Gateaux derivatives $\partial_{g_i} G_i(\varepsilon, \boldsymbol{\rho}, \boldsymbol{x}, \boldsymbol{g})h: \left(-\varepsilon_0, \varepsilon_0\right)\times I^m \times     B_{\boldsymbol{\rho}_0}\left(\boldsymbol{x}_0\right) \times V_\mu\to Y^{k-1}$ and $\partial_{g_j}     G_i(\varepsilon, \boldsymbol{\rho}, \boldsymbol{x}, \boldsymbol{g})h: \left(-\varepsilon_0,    \varepsilon_0\right)\times I^m \times B_{\boldsymbol{\rho}_0}\left(\boldsymbol{x}_0\right) \times V_\mu\to Y^{k-1}$ exist and are continuous.
\end{lemma}

\begin{proof}
We claim $\partial_{ g_i} G_{i,1}(\varepsilon, g_i)h=\partial_{ g_i} G_{i11}+\partial_{ g_i} G_{i12}+\partial_{ g_i} G_{i13}+\partial_{ g_i} G_{i14}$ is continuous, where
	\begin{equation}\label{2-11}
		\begin{split}		
    			&\partial_{ g_i} G_{i11}=C_\gamma\int\!\!\!\!\!\!\!\!\!\; {}-{} \frac{h(\eta)\sin(\beta-\eta)d\eta}{\left( |\varepsilon|^{2+2\gamma}\left(g_i(\beta)-g_i(\eta)\right)^2+4(\rho_i+\varepsilon|\varepsilon|^\gamma g_i(\beta))(\rho_i+\varepsilon|\varepsilon|^\gamma g_i(\eta))\sin^2\left(\frac{\beta-\eta}{2}\right)\right)^{\frac{\gamma}{2}}}\\
			& \ \ \ \ -\frac{\gamma C_\gamma}{2}\int\!\!\!\!\!\!\!\!\!\; {}-{} \frac{(\rho_i+\varepsilon|\varepsilon|^\gamma  g_i(\eta))\sin(\beta-\eta)}{\left( |\varepsilon|^4\left(g_i(\beta)-g_i(\eta)\right)^2+4(\rho_i+\varepsilon|\varepsilon|^\gamma g_i(\beta))(\rho_i+\varepsilon|\varepsilon|^\gamma  g_i(\eta))\sin^2\left(\frac{\beta-\eta}{2}\right)\right)^{\frac{\gamma+2}{2}}}\\
			& \ \ \ \ \times\bigg(\varepsilon|\varepsilon|^\gamma(2(g_i(\beta)-g_i(\eta))(h(\beta)-h(\eta))\\
			& \ \ \ \ \ \ \ \ +4(h(\beta)(\rho_i+\varepsilon|\varepsilon|^\gamma  g_i(\eta))+h(\eta)(\rho_i+\varepsilon|\varepsilon|^\gamma g_i(\beta)))\sin^2\left(\frac{\beta-\eta}{2}\right)\bigg)d\eta,
		\end{split}
	\end{equation}
    \begin{equation}\label{2-12}
    	\begin{split}
    		&\partial_{ g_i} G_{i12}=C_\gamma\int\!\!\!\!\!\!\!\!\!\; {}-{} \frac{(h'(\eta)-h'(\beta))\cos(\beta-\eta)d\eta}{\left( |\varepsilon|^{2+2\gamma}\left(g_i(\beta)-g_i(\eta)\right)^2+4(\rho_i+\varepsilon|\varepsilon|^\gamma g_i(\beta))(\rho_i+\varepsilon|\varepsilon|^\gamma g_i(\eta))\sin^2\left(\frac{\beta-\eta}{2}\right)\right)^{\frac{\gamma}{2}}}\\
    		&  \ \ \ -\frac{\gamma C_\gamma\varepsilon|\varepsilon|^\gamma}{2}\int\!\!\!\!\!\!\!\!\!\; {}-{} \frac{(g'_i(\eta)-g'_i(\beta))\cos(\beta-\eta)}{\left( |\varepsilon|^{2+2\gamma}\left(g_i(\beta)-g_i(\eta)\right)^2+4(\rho_i+\varepsilon|\varepsilon|^\gamma g_i(\beta))(\rho_i+\varepsilon|\varepsilon|^\gamma g_i(\eta))\sin^2\left(\frac{\beta-\eta}{2}\right)\right)^{\frac{\gamma+2}{2}}}\\
    			&  \ \ \ \times\bigg(\varepsilon|\varepsilon|^\gamma(2(g_i(\beta)-g_i(\eta))(h(\beta)-h(\eta))\\
			& \ \ \ \ \ \ \ \ +4(h(\beta)(\rho_i+\varepsilon|\varepsilon|^\gamma  g_i(\eta))+h(\eta)(\rho_i+\varepsilon|\varepsilon|^\gamma g_i(\beta)))\sin^2\left(\frac{\beta-\eta}{2}\right)\bigg)d\eta,
    	\end{split}
    \end{equation}
    \begin{equation}\label{2-13}
    	\begin{split}
    		&\partial_{ g_i} G_{i13}={\resizebox{.9\hsize}{!}{$\frac{C_\gamma \varepsilon|\varepsilon|^\gamma h'(\beta)}{\rho_i+\varepsilon|\varepsilon|^\gamma g_i(\beta)}\int\!\!\!\!\!\!\!\!\!\; {}-{} \frac{(g_i(\beta)-g_i(\eta))\cos(\beta-\eta)d\eta}{\left( |\varepsilon|^{2+2\gamma}\left(g_i(\beta)-g_i(\eta)\right)^2+4(\rho_i+\varepsilon|\varepsilon|^\gamma g_i(\beta))(\rho_i+\varepsilon|\varepsilon|^\gamma g_i(\eta))\sin^2\left(\frac{\beta-\eta}{2}\right)\right)^{\frac{\gamma}{2}}}$}}\\
    		&
            \ \ \ \ +{\resizebox{.93\hsize}{!}{$\frac{C_\gamma \varepsilon|\varepsilon|^\gamma g'_i(\beta)}{\rho_i+\varepsilon|\varepsilon|^\gamma g_i(\beta)}\int\!\!\!\!\!\!\!\!\!\; {}-{} \frac{(h(\beta)-h(\eta))\cos(\beta-\eta)d\eta}{\left( |\varepsilon|^{2+2\gamma}\left(g_i(\beta)-g_i(\eta)\right)^2+4(\rho_i+\varepsilon|\varepsilon|^\gamma g_i(\beta))(\rho_i+\varepsilon|\varepsilon|^\gamma g_i(\eta))\sin^2\left(\frac{\beta-\eta}{2}\right)\right)^{\frac{\gamma}{2}}}$}}\\
            &
            \ \ \ \ -{\resizebox{.9\hsize}{!}{$\frac{C_\gamma |\varepsilon|^{2+2\gamma} g'_i(\beta)h(\beta)}{(\rho_i+\varepsilon|\varepsilon|^\gamma
            g_i(\beta))^2}\int\!\!\!\!\!\!\!\!\!\; {}-{} \frac{(g_i(\beta)-g_i(\eta))\cos(\beta-\eta)d\eta}{\left( |\varepsilon|^{2+2\gamma}\left(g_i(\beta)-g_i(\eta)\right)^2+4(\rho_i+\varepsilon|\varepsilon|^\gamma g_i(\beta))(\rho_i+\varepsilon|\varepsilon|^\gamma g_i(\eta))\sin^2\left(\frac{\beta-\eta}{2}\right)\right)^{\frac{\gamma}{2}}}$}}\\
    		&
            \ \ \ \ -{\resizebox{.9\hsize}{!}{$\frac{\gamma C_\gamma |\varepsilon|^{2+2\gamma} g'_i(\beta)}{2(\rho_i+\varepsilon|\varepsilon|^\gamma g_i(\beta))}\int\!\!\!\!\!\!\!\!\!\; {}-{} \frac{(g_i(\beta)-g_i(\eta))\cos(\beta-\eta)}{\left( |\varepsilon|^{2+2\gamma}\left(g_i(\beta)-g_i(\eta)\right)^2+4(\rho_i+\varepsilon|\varepsilon|^\gamma g_i(\beta))(\rho_i+\varepsilon|\varepsilon|^\gamma g_i(\eta))\sin^2\left(\frac{\beta-\eta}{2}\right)\right)^{\frac{\gamma+2}{2}}}$}}\\
    		&
            \ \ \ \ \times\bigg(\varepsilon|\varepsilon|^\gamma(2(g_i(\beta)-g_i(\eta))(h(\beta)-h(\eta))\\
			& \ \ \ \ \ \ \ \ +4(h(\beta)(\rho_i+\varepsilon|\varepsilon|^\gamma  g_i(\eta))+h(\eta)(\rho_i+\varepsilon|\varepsilon|^\gamma g_i(\beta)))\sin^2\left(\frac{\beta-\eta}{2}\right)\bigg)d\eta
    	\end{split}
    \end{equation}
    and
    \begin{equation}\label{2-14}
    	\begin{split}
    		&\partial_{ g_i} G_{i14}={\resizebox{.9\hsize}{!}{$\frac{C_\gamma |\varepsilon|^{1+\gamma}}{\rho_i+\varepsilon|\varepsilon|^\gamma g_i(\beta)}\int\!\!\!\!\!\!\!\!\!\; {}-{} \frac{(h'(\eta)g'_i(\beta)+h'(\beta)g'_i(\eta))\sin(\beta-\eta)d\eta}{\left( |\varepsilon|^{2+2\gamma}\left(g_i(\beta)-g_i(\eta)\right)^2+4(\rho_i+\varepsilon|\varepsilon|^\gamma g_i(\beta))(\rho_i+\varepsilon|\varepsilon|^\gamma g_i(\eta))\sin^2\left(\frac{\beta-\eta}{2}\right)\right)^{\frac{\gamma}{2}}}$}}\\
    		& \ \ \ \ -{\resizebox{.93\hsize}{!}{$\frac{\gamma C_\gamma |\varepsilon|^{2+2\gamma}}{2(\rho_i+\varepsilon|\varepsilon|^\gamma g_i(\beta))}\int\!\!\!\!\!\!\!\!\!\; {}-{} \frac{g'_i(\beta)g'_i(\eta)\sin(\beta-\eta)}{\left( |\varepsilon|^{2+2\gamma}\left(g_i(\beta)-g_i(\eta)\right)^2+4(\rho_i+\varepsilon|\varepsilon|^\gamma g_i(\beta))(\rho_i+\varepsilon|\varepsilon|^\gamma g_i(\eta))\sin^2\left(\frac{\beta-\eta}{2}\right)\right)^{\frac{\gamma+2}{2}}}$}}\\
    		& \ \ \ \ \times\bigg(\varepsilon|\varepsilon|^\gamma(2(g_i(\beta)-g_i(\eta))(h(\beta)-h(\eta))\\
			& \ \ \ \ \ \ \ \ +4(h(\beta)(\rho_i+\varepsilon|\varepsilon|^\gamma  g_i(\eta))+h(\eta)(\rho_i+\varepsilon|\varepsilon|^\gamma g_i(\beta)))\sin^2\left(\frac{\beta-\eta}{2}\right)\bigg)d\eta\\
			&\ \ \ \ \ \ \ \ \ {\resizebox{.93\hsize}{!}{$-\frac{|\varepsilon|^{2+2\gamma}C_{\gamma}g'_{i}(\beta)}{(\rho_{i}+\varepsilon|\varepsilon|^{\gamma}g_{i}(\beta))^{2}}\int\!\!\!\!\!\!\!\!\!\; {}-{} \frac{g'_{i}(\eta)h(\beta)\sin{(\beta-\eta)}}{\left( |\varepsilon|^{2+2\gamma}\left(g_i(\beta)-g_i(\eta)\right)^2+4(\rho_i+\varepsilon|\varepsilon|^\gamma g_i(\beta))(\rho_i+\varepsilon|\varepsilon|^\gamma g_i(\eta))\sin^2\left(\frac{\beta-\eta}{2}\right)\right)^{\frac{\gamma}{2}}}\, d\eta.$}}
    	\end{split}
    \end{equation}
    Moreover,
    \begin{equation}\label{partial_i2}
  \partial_{ g_i} G_{i,2}(\varepsilon, \boldsymbol{\rho}, \boldsymbol{x}, \boldsymbol{g}) = O(\varepsilon),
\end{equation}
\begin{equation}\label{partial_i3}
  \partial_{ g_i} G_{i,3}(\varepsilon, \boldsymbol{\rho}, \boldsymbol{x}, \boldsymbol{g}) = O(\varepsilon),
\end{equation}
and
\begin{equation}\label{partial_j2}
  \partial_{ g_j} G_{i}(\varepsilon, \boldsymbol{\rho}, \boldsymbol{x}, \boldsymbol{g}) = O(\varepsilon) .
\end{equation}
    To this aim, the first step is to show
    \begin{equation*}
        \lim\limits_{t\to0}\left\|\frac{G_{i1l}(\varepsilon, g_i+th)-G_{i1l}(\varepsilon, g_i)}{t}-\partial_{g_i}G_{i1l}(\varepsilon, g,h)\right\|_{Y^{k-1}}\to 0
    \end{equation*}
	for $l=1,\cdots,4$. For the sake of simplicity, we will focus exclusively on the most singular scenario, which is when $l=2$, and utilize the notations introduced in Lemma \ref{2-2}. In this context, the following holds
	\begin{equation*}
		\begin{split}
			&\frac{G_{i12}(\varepsilon, g_i+th)-G_{i12}(\varepsilon, g_i)}{t}-\partial_{g_i}G_{i12}(\varepsilon, g_i,h)\\
			&=\frac{1}{t}\int\!\!\!\!\!\!\!\!\!\; {}-{}{\resizebox{.9\hsize}{!}{$(g'_i(\beta)-g'_i(\eta))\cos(\beta-\eta)\bigg(\frac{1}{D_\gamma(g_i+th)^{\gamma/2}}-\frac{1}{D_\gamma(g_i)}+t\frac{\Delta g_i\Delta h+2(h\tilde{R}_i+\tilde{h}R_i)\sin^2(\frac{\beta-\eta}{2})}{D_\gamma(g_i)}\bigg)d\eta$}}\\
			&\ \ \ \ +\int\!\!\!\!\!\!\!\!\!\; {}-{}(h'(\beta)-h'(\eta))\cos(\beta-\eta)\bigg(\frac{1}{D_\gamma(g_i+th)}-\frac{1}{D_\gamma(g_i)}\bigg)d\eta\\
			&=F_{21}+F_{22}.
		\end{split}
	\end{equation*}
   By performing $\partial^{k-1}$ derivatives of $G_{21}$, we deduce
    \begin{equation*}
    	\begin{split}
    	  \partial^{k-1}G_{i12}&=\frac{1}{t}\int\!\!\!\!\!\!\!\!\!\; {}-{}
            \bigg(\frac{1}  {D_\gamma(g_i+th)^{\gamma/2}}-\frac{1}{D_\gamma(g_i)^{\gamma/2}}+t\frac{\Delta g_i\Delta     h+2(\tilde{R}_i h+\tilde{h} \rho_i)\sin^2(\frac{\beta-\eta}{2})}{D_\gamma(g_i)^{\gamma/2}}\bigg)\\
    	  & \ \ \ \ \times(\partial^kg_i(\beta)-\partial^kg_i(\eta))\cos(\beta-\eta)d\eta+l.o.t.
    	\end{split}
    \end{equation*}
    By applying the mean value theorem, we obtain
    \begin{equation*}
    	\frac{1}{D_\gamma(g_i+th)^{\gamma/2}}-\frac{1}{D_\gamma(g_i)^{\gamma/2}}+t\frac{\Delta g_i\Delta h+2(\tilde{R}_i h+\tilde{h} \rho_i)\sin^2(\frac{\beta-\eta}{2})}{D_\gamma(g_i)^{\gamma/2}}\sim \frac{Ct^2}{|\sin(\frac{\beta-\eta}{2})|^{\gamma/2}}\varphi(\varepsilon,g_i,h),
    \end{equation*}
    with $\|\varphi(\varepsilon,g_i,h)\|_{L^\infty}<\infty$. Hence, we can conclude that
    \begin{equation*}
    	\|G_{i12}\|_{Y^{k-1}}\le Ct\left\|\int\!\!\!\!\!\!\!\!\!\; {}-{}\frac{\partial^kg_i(\beta)-\partial^kg_i(\eta)}{|\sin(\frac{\beta-\eta}{2})|^{\gamma/2}}d\eta\right\|_{L^2}\le Ct\|f\|_{X^{k+\gamma-1}}.
    \end{equation*}
    Using the same reasoning employed in \eqref{2-7}, we can similarly establish that $\norm{F_{22}}_{Y^{k-1}}$ is bounded by $Ct\norm{f}_{X^{k+\gamma-1}}$. Thus, we conclude the first step by allowing $t$ to approach zero. The second step involves demonstrating the continuity of $\partial_{g_i} G_{i,1}(\varepsilon, f)$, a reliance on \eqref{2-7}.

Applying the same approach as described earlier, we infer that
    \begin{equation}\label{2-15}
    	\partial_{g_i} G_{i,2}(\varepsilon, g_i)=|\varepsilon|\partial_{g_i}\mathcal{R}_{i,2}(\varepsilon,g_i)
    \end{equation}
    and
       \begin{equation}\label{deriv_fi3}
    	\partial_{g_i} G_{i,3}(\varepsilon, g_i)=|\varepsilon|\partial_{g_i}\mathcal{R}_{i,3}(\varepsilon,g_i)
    \end{equation}
    are continuous, where $\mathcal{R}_{i,2}(\varepsilon,g_i)$ and $\mathcal{R}_{i,3}(\varepsilon,g_i)$ are the same terms as in \eqref{333} and \eqref{3-6} respectively. This concludes the proof of Lemma \ref{lem2-3}.
\end{proof}

From \eqref{2-11}-\eqref{deriv_fi3}, by letting $\varepsilon=0$ and $g_i\equiv 0$, one has
\begin{equation}\label{gateaux}
	\partial_{g_i} G_{i}(0,0)h=\frac{C_\gamma}{\rho_i^\gamma}\left(1-\frac{\gamma}{2}\right)\int\!\!\!\!\!\!\!\!\!\; {}-{} \frac{h(\beta-\eta)\sin(\eta)d\eta}{\left(4\sin^2(\frac{\eta}{2})\right)^{\frac{\gamma}{2}}}-\frac{C_\gamma}{\rho_i^\gamma}\int\!\!\!\!\!\!\!\!\!\; {}-{} \frac{(h'(\beta)-h'(\beta-\eta))\cos(\eta)d\eta}{\left(4\sin^2(\frac{\eta}{2})\right)^{\frac{\gamma}{2}}} .
\end{equation}

It is evident that the functional $F_i$ defined in \eqref{Gi_eq} possess $C^1$ regularity. Let $\boldsymbol{G}=\left(G_1, \ldots, G_m\right)$.   Having established the $C^1$ regularity of the functional $G_i$, we can derive the following result: Finding roots of the functional $\boldsymbol{G} = 0$ is equivalent to identifying critical points $\boldsymbol{x}_0$ of the Kirchhoff-Routh function  $\mathcal{W}_m$. As a result, we can produce a collection of trivial solutions for \eqref{1-1} with $\gamma \in (1,2)$.

\begin{lemma}\label{equivalence}
  The equation $\boldsymbol{G}(0, \boldsymbol{\rho}, \boldsymbol{x}, \mathbf{0})=\mathbf{0}$ is equivalent to
  \[\nabla \mathcal{W}_m(\boldsymbol{x})=\mathbf{0} .\]
  Therefore, it can be concluded that,  for any $\boldsymbol{\rho}$, it holds that $\boldsymbol{G}\left(0, \boldsymbol{\rho}, \boldsymbol{x}_0, \mathbf{0}\right) \equiv \mathbf{0}$.
\end{lemma}

\begin{proof}
  This lemma can be directly inferred by considering the expansions of the functional given
 by \eqref{3-7}, along with the definition of the Kirchhoff-Routh function $\mathcal{W}_m.$
\end{proof}
\section{Linearization and isomorphism}\label{section4}
In this part, our primary focus lies in analyzing the linearization of the functional introduced in Section \ref{2}. We know from Lemma \eqref{equivalence} that  $(0,\boldsymbol{\rho},\boldsymbol{x},0)$ constitutes a solution to $\boldsymbol{G}=\left(G_1, \ldots, G_m\right)=0$ if and only if $\boldsymbol{x}$ represents a critical point of  $\mathcal{W}_m$. Let's proceed by selecting $\boldsymbol{x}_0$ as one such critical point of $\mathcal{W}_m$, which implies that $(0,\boldsymbol{\rho}, \boldsymbol{x}_0,0)$ satisfies $\boldsymbol{G}=0$. We will proceed to investigate the linearization of $\boldsymbol{G}$ at $(0,\boldsymbol{\rho},\boldsymbol{x},0)$.

As indicated in \eqref{3-7} at the conclusion of the proof of Lemma \ref{p3-1}, when $\varepsilon=0$ and $g_i\equiv 0$ for all $i=1,\ldots,m$, the Gateaux derivatives are:
\begin{equation}\label{4-1}
    \left\{
    \begin{aligned}
        &\partial_{g_i}G_{i}(0,\boldsymbol{\rho},\boldsymbol{x},0) h=\dfrac{C_\gamma}  {\rho_i^\gamma}\left(1-\dfrac{\gamma}{2}\right) \displaystyle\int\!\!\!\!\!\!\!\!\!\; {}-{}\frac{h(\beta- \eta)\sin(\eta)d\eta}{\left(4\sin^2(\frac{\eta}   {2})\right)^{\frac{\gamma}{2}}}-\frac{C_\gamma}      {\rho_i^\gamma}\displaystyle\int\!\!\!\!\!\!\!\!\!\; {}-{}       \frac{(h'(\beta)-h'(\beta-\eta))\cos(\eta)d\eta}      {\left(4\sin^2(\frac{\eta}{2})\right)^{\frac{\gamma}{2}}},\\
        &\partial_{g_j}G_{i}(0,\boldsymbol{\rho},\boldsymbol{x},0) h=0,\,\,\,j\not=i.
    \end{aligned}
    \right.
\end{equation}


\begin{lemma}\label{iso}
    Let $h=\sum_{j=1}^{\infty}\left(a_j \cos (j \beta)+b_j \sin (j \beta)\right)$ be in $ X^{k+\gamma-1}$. Then it holds
    \begin{equation*}
        \begin{aligned}
            & \partial_{g_i} G_i(0, \boldsymbol{\rho}, \boldsymbol{x}, \mathbf{0}) h=\frac{1}{ \rho_i^\gamma} \sum_{j=1}^{\infty} \sigma_j j\left(a_j \sin (j \beta)-b_j \cos (j \beta)\right) \\
            & \partial_{g_j} G_i(0, \boldsymbol{\rho}, \boldsymbol{x}, \mathbf{0}) h=0, \quad j \neq i
        \end{aligned}
    \end{equation*}
    with
    \begin{equation*}
        \sigma_j=
        \begin{cases}
            2^{\gamma-1}\dfrac{ \Gamma(1-\gamma)} {(\Gamma(1-\frac{\gamma}{2}))^2}      \left(\dfrac{\Gamma(1+\frac{\gamma}{2})}{\Gamma(2-\frac{\gamma}{2})}-\dfrac{\Gamma(j+\frac{\gamma}{2})}{\Gamma(1+j-\frac{\gamma}{2})}\right), & j \geq 2, \\
            0, & j=1 .
        \end{cases}
    \end{equation*}
    Furthermore, for $\boldsymbol{\rho} \in I^m, \boldsymbol{x} \in \Omega^{m},$ the operator $\partial_{g_{i}} G_{i}(0, \boldsymbol{\rho}, \boldsymbol{x}, \mathbf{0}): X^{k+\gamma-1} \rightarrow Y_0^{k-1}$ is an isomorphism. Additionally, the operator $\partial_{\boldsymbol{g}} \boldsymbol{G}(0, \boldsymbol{\rho}, \boldsymbol{x}, \mathbf{0}): (X^{k+\gamma-1})^m \rightarrow (Y_0^{k-1})^m$ is also an isomorphism.
\end{lemma}

\begin{proof}
    Direct computations obtained in \eqref{gateaux} gives
    \begin{equation*}
        \begin{aligned}
            & \partial_{g_i} G_i(0, \boldsymbol{\rho}, \boldsymbol{x}, \mathbf{0}) h
            =\frac{C_\gamma}{\rho_i^\gamma}\left(1-\frac{\gamma}{2}\right)\int\!\!\!\!\!\!\!\!\!\; {}-{} \frac{h(\beta-\eta)\sin(\eta)d\eta}{\left(4\sin^2(\frac{\eta}{2})\right)^{\frac{\gamma}{2}}}-\frac{C_\gamma}{\rho_i^\gamma}\int\!\!\!\!\!\!\!\!\!\; {}-{} \frac{(h'(\beta)-h'(\beta-\eta))\cos(\eta)d\eta}{\left(4\sin^2(\frac{\eta}{2})\right)^{\frac{\gamma}{2}}} .
        \end{aligned}
    \end{equation*}
   According to Proposition 2.4 in \cite{Cas1}, it follows that
    \begin{equation*}
        \partial_{g_i} G_i(0, \boldsymbol{\rho}, \boldsymbol{x}, \mathbf{0})(\cos(j\beta))=\frac{1}{\rho_i^\gamma} \sigma_j  j\sin (j \beta) .
    \end{equation*}
    Now, let's address the second term in the Fourier expansion of $h$. Since $K_\gamma^0(\beta)$ can be expressed as
    \begin{equation*}
        h(\beta)=\sum_{j=1}^{\infty}\left(a_j \cos (j \beta)+b_j \sin (j \beta)\right),
    \end{equation*}
    we will compute that $\partial_{g_i} G_i(0, \boldsymbol{\rho}, \boldsymbol{x}, \mathbf{0})(\sin(j\beta))$ acts on $-\frac{1}{\rho_i^\gamma} \sigma_j  j\cos (j \beta) $ for each $j\ge 2$. By \eqref{4-1} we have
	\begin{equation}\label{sin}
		\frac{C_\gamma}{\rho_i^\gamma}\left(1-\frac{\gamma}{2}\right)\int\!\!\!\!\!\!\!\!\!\; {}-{} \frac{b_j \sin(j\beta-j\eta)\sin(\eta)d\eta}{\left(4\sin^2(\frac{\eta}{2})\right)^{\frac{\gamma}{2}}}-\frac{C_\gamma}{\rho_i^\gamma} j\int\!\!\!\!\!\!\!\!\!\; {}-{} \frac{(b_j\cos(j\beta)-b_j\cos(j\beta-j\eta))\cos(\eta)d\eta}{\left(4\sin^2(\frac{\eta}{2})\right)^{\frac{\gamma}{2}}}.
	\end{equation}
	Using identities
    \begin{equation*}
        \cos\left(\frac{y}{2}\right)\left|\sin\left(\frac{y}{2}\right)\right|^{1-\gamma}=\frac{2}{2-\gamma}\partial_y\left(\left|\sin\left(\frac{y}{2}\right)\right|^{2-\gamma}\right),
    \end{equation*}
    and  the following identity (see \cite{magnus})
    \begin{equation}\label{2-16}
        \int_0^\pi(\sin(\eta))^{2-\gamma}e^{\text{i} j\eta}d\eta=\frac{\pi e^{j\pi\text i}\Gamma(3-\gamma)}{2^{2-\gamma}\Gamma(2+j-\frac{\gamma}{2})\Gamma(2-j-\frac{\gamma}{2})}, \ \ \ \ \forall \, \gamma<3, \ \ \forall \, j\in\mathbb{R}.
    \end{equation}
	By applying integration by parts to the first term in \eqref{sin}, we obtain
	\begin{equation}\label{2-17}
		\begin{split}
			&\frac{C_\gamma}{\rho_i^\gamma}\left(1-\frac{\gamma}{2}\right)\int\!\!\!\!\!\!\!\!\!\; {}-{} \frac{b_j\sin(j\beta-j\eta)\sin(\eta)d\eta}{\left(4\sin^2(\frac{\eta}{2})\right)^{\frac{\gamma}{2}}}\\
			&=2^{1-\gamma}\frac{C_\gamma}{\rho_i^\gamma}\left(1-\frac{\gamma}{2}\right)\int\!\!\!\!\!\!\!\!\!\; {}-{}b_j\sin(j\beta-j\eta)\cos\left(\frac{\eta}{2}\right)\left|\sin\left(\frac{y}{2}\right)\right|^{1-\gamma}d\eta\\
			&=2^{1-\gamma}\frac{C_\gamma}{\rho_i^\gamma}\left(1-\frac{\gamma}{2}\right)\frac{2j}{2-\gamma}b_j\int\!\!\!\!\!\!\!\!\!\; {}-{} \cos(j\beta-j\eta)\left|\sin\left(\frac{y}{2}\right)\right|^{2-\gamma}d\eta\\
			&=2^{1-\gamma}\frac{C_\gamma}{\rho_i^\gamma}\left(1-\frac{\gamma}{2}\right)\frac{2j}{2-\gamma}b_j\cos(j\beta)\int\!\!\!\!\!\!\!\!\!\; {}-{} \cos(j\eta)\left|\sin\left(\frac{y}{2}\right)\right|^{2-\gamma}d\eta\\
			&=2^{1-\gamma}\frac{C_\gamma}{\rho_i^\gamma}\left(1-\frac{\gamma}{2}\right)\frac{2j}{2-\gamma}\frac{b_j}{2\pi}\frac{\pi \cos(j\pi)\Gamma(3-\gamma)}{2^{1-\gamma}\Gamma(2+j-\frac{\gamma}{2})\Gamma(2-j-\frac{\gamma}{2})}\cos(j\beta).
		\end{split}
	\end{equation}
    For the second term in \eqref{sin} and the identity $\cos \eta=1-2\sin^2(\frac{\eta}{2})$, it holds
    \begin{equation*}
    	\begin{split}
    		&-\frac{C_\gamma}{\rho_i^\gamma} j b_j\int\!\!\!\!\!\!\!\!\!\; {}-{} \frac{(\cos(j\beta)-\cos(j\beta-j\eta))\cos(\eta)d\eta}{\left(4\sin^2(\frac{\eta}{2})\right)^{\frac{\gamma}{2}}}\\
    		&=-2^{-\gamma}\frac{C_\gamma}{\rho_i^\gamma} j b_j\int\!\!\!\!\!\!\!\!\!\; {}-{} (\cos(j\beta)-\cos(j\beta-j\eta))\left|\sin\left(\frac{\eta}{2}\right)\right|^{-\gamma}d\eta\\
    		&\ \ \ \ +2^{1-\gamma}\frac{C_\gamma}{\rho_i^\gamma} j b_j\int\!\!\!\!\!\!\!\!\!\; {}-{} (\cos(j\beta)-\cos(j\beta-j\eta))\left|\sin\left(\frac{\eta}{2}\right)\right|^{2-\gamma}d\eta.
    	\end{split}
    \end{equation*}
    In accordance with \cite[Lemma 2.6]{Cas1} and the identity \eqref{2-16}, we can express the previous equation as
    \begin{equation}\label{2-18}
    	\begin{split}
    		& \ \ \ \ -\frac{C_\gamma}{2\pi \rho_i^\gamma} j b_j\frac{2\pi \Gamma(1-\gamma)}{\Gamma(\frac{\gamma}{2})\Gamma(1-\frac{\gamma}{2})}\left(\frac{\Gamma(\frac{\gamma}{2})}{\Gamma(1-\frac{\gamma}{2})}-\frac{\Gamma(j+\frac{\gamma}{2})}{\Gamma(1+j-\frac{\gamma}{2})}\right)\cos(j\beta)\\
    		&+2^{-1}\frac{C_\gamma}{2\pi \rho_i^\gamma} j b_j\frac{2\pi \Gamma(3-\gamma)}{\Gamma(\frac{\gamma}{2}-1)\Gamma(2-\frac{\gamma}{2})}\left(\frac{\Gamma(\frac{\gamma}{2}-1)}{\Gamma(2-\frac{\gamma}{2})}-\frac{\Gamma(j-1+\frac{\gamma}{2})}{\Gamma(2+j-\frac{\gamma}{2})}\right)\cos(j\beta)
    	\end{split}
    \end{equation}
    We incorporate the previously obtained results, particularly equations \eqref{2-17} and \eqref{2-18}, into equation \eqref{sin}. Next, we isolate the contribution of the $k$th mode in the following manner
    \begin{equation}\label{s-1}
        \begin{aligned}
            &2^{1-\gamma}\left(1-\frac{\gamma}{2}\right) \frac{C_\gamma}{2\pi \rho_i^\gamma} \frac{2 j}{2-\gamma} b_j \frac{ \pi \cos (j \pi) \Gamma(3-\gamma)}{2^{1-\gamma} \Gamma\left(2+j-\frac{\gamma}{2}\right) \Gamma\left(2-j-\frac{\gamma}{2}\right)} \\
            & - b_j 2^{-\gamma} \frac{C_\gamma}{2\pi \rho_i^\gamma}j 2^\gamma \frac{2 \pi \Gamma(1-\gamma)}{\Gamma\left(\frac{\gamma}{2}\right) \Gamma\left(1-\frac{\gamma}{2}\right)}\left(\frac{\Gamma\left(\frac{\gamma}{2}\right)}{\Gamma\left(1-\frac{\gamma}{2}\right)}-\frac{\Gamma\left(j+\frac{\gamma}{2}\right)}{\Gamma\left(1+j-\frac{\gamma}{2}\right)}\right) \\
            & +b_j 2^{-\gamma+1} \frac{C_\gamma}{2\pi \rho_i^\gamma} j 2^{\gamma-2} \frac{2 \pi \Gamma(3-\gamma)}{\Gamma\left(\frac{\gamma}{2}-1\right) \Gamma\left(2-\frac{\gamma}{2}\right)}\left(\frac{\Gamma\left(\frac{\gamma}{2}-1\right)}{\Gamma\left(2-\frac{\gamma}{2}\right)}-\frac{\Gamma\left(j-1+\frac{\gamma}{2}\right)}{\Gamma\left(2+j-\frac{\gamma}{2}\right)}\right)\\
            &=:s_1+s_2+s_3+s_4+s_5.
        \end{aligned}
    \end{equation}
    Now by straightforward computation together with  the definition of $C_\gamma=\frac{\Gamma (\frac{\gamma}{2})}{2^{1-\gamma} \Gamma(1-\frac{\gamma}{2})}$ and some properties of the Gamma function. We can group the $s_2$ and $s_4$ into
    \begin{equation}\label{s-2}
        \begin{aligned}
            s_2+s_4=& -b_j 2^{-\gamma} \frac{C_\gamma}{2\pi \rho_i^\gamma}j 2^\gamma \frac{2 \pi \Gamma(1-\gamma)}{\Gamma\left(\frac{\gamma}{2}\right) \Gamma\left(1-\frac{\gamma}{2}\right)}\left(\frac{\Gamma\left(\frac{\gamma}{2}\right)}{\Gamma\left(1-\frac{\gamma}{2}\right)}\right)\\
            & \qquad+b_j 2^{-\gamma+1} \frac{C_\gamma}{2\pi \rho_i^\gamma} j 2^{\gamma-2} \frac{2 \pi \Gamma(3-\gamma)}{\Gamma\left(\frac{\gamma}{2}-1\right) \Gamma\left(2-\frac{\gamma}{2}\right)}\left(\frac{\Gamma\left(\frac{\gamma}{2}-1\right)}{\Gamma\left(2-\frac{\gamma}{2}\right)}\right) \\
            &= -\frac{b_j}{ \rho_i^\gamma} 2^{-\gamma} \frac{1}{2\pi}\frac{\Gamma (\frac{\gamma}{2})}{2^{1-\gamma} \Gamma(1-\frac{\gamma}{2})}j 2^\gamma \frac{2 \pi \Gamma(1-\gamma)}{\Gamma\left(\frac{\gamma}{2}\right) \Gamma\left(1-\frac{\gamma}{2}\right)}\left(\frac{\Gamma\left(\frac{\gamma}{2}\right)}{\Gamma\left(1-\frac{\gamma}{2}\right)}\right) \\
            &\qquad+\frac{b_j}{\rho_i^\gamma} 2^{-\gamma+1} \frac{1}{2\pi}\frac{\Gamma (\frac{\gamma}{2})}{2^{1-\gamma} \Gamma(1-\frac{\gamma}{2})} j 2^{\gamma-2} \frac{2 \pi \Gamma(3-\gamma)}{ \Gamma\left(2-\frac{\gamma}{2}\right)}\left(\frac{1}{\Gamma\left(2-\frac{\gamma}{2}\right)}\right) \\
            & =-j \frac{b_j}{\rho_i^\gamma} 2^{\gamma-1} \frac{\Gamma(1-\gamma)}{\left(\Gamma\left(1-\frac{\gamma}{2}\right)\right)^2}\left(\frac{\Gamma\left(\frac{\gamma}{2}\right)}{\Gamma\left(1-\frac{\gamma}{2}\right)}-\frac{1}{2} \frac{(2-\gamma)(1-\gamma) \Gamma\left(\frac{\gamma}{2}\right)}{\left(1-\frac{\gamma}{2}\right) \Gamma\left(2-\frac{\gamma}{2}\right)}\right)\\
            & =-j \frac{b_j}{\rho_i^\gamma}2^{\gamma-1} \frac{\Gamma(1-\gamma)}{\left(\Gamma\left(1-\frac{\gamma}{2}\right)\right)^2} \frac{\Gamma\left(1+\frac{\gamma}{2}\right)}{\Gamma\left(2-\frac{\gamma}{2}\right)}\left(\frac{1-\frac{\gamma}{2}}{\frac{\gamma}{2}}-\frac{1}{2} \frac{(2-\gamma)(1-\gamma)}{\left(1-\frac{\gamma}{2}\right)\left(\frac{\gamma}{2}\right)}\right) \\
            & =-j \frac{b_j}{\rho_i^\gamma} 2^{\gamma-1} \frac{\Gamma(1-\gamma)}{\left(\Gamma\left(1-\frac{\gamma}{2}\right)\right)^2} \frac{\Gamma\left(1+\frac{\gamma}{2}\right)}{\Gamma\left(2-\frac{\gamma}{2}\right)}.
        \end{aligned}
    \end{equation}
    Similarly we proceed with $s_3$
    \begin{equation}\label{s-3}
        \begin{aligned}
            s_3&=b_j 2^{-\gamma} \frac{C_\gamma}{2\pi \rho_i^\gamma}j 2^\gamma \frac{2 \pi \Gamma(1-\gamma)}{\Gamma\left(\frac{\gamma}{2}\right) \Gamma\left(1-\frac{\gamma}{2}\right)}\frac{\Gamma\left(j+\frac{\gamma}{2}\right)}{\Gamma\left(1+j-\frac{\gamma}{2}\right)}\\
            &=j \frac{b_j}{\rho_i^\gamma} 2^{-\gamma} \frac{\Gamma (\frac{\gamma}{2})}{2\pi 2^{1-\gamma} \Gamma(1-\frac{\gamma}{2})}2^{\gamma} \frac{2 \pi \Gamma(1-\gamma)}{\Gamma\left(\frac{\gamma}{2}\right) \Gamma\left(1-\frac{\gamma}{2}\right)}\frac{\Gamma\left(j+\frac{\gamma}{2}\right)}{\Gamma\left(1+j-\frac{\gamma}{2}\right)}\\
            &=j \frac{b_j}{\rho_i^\gamma} 2^{\gamma-1}  \frac{ \Gamma(1-\gamma)}{ \left(\Gamma\left(1-\frac{\gamma}{2}\right)\right)^2}\frac{\Gamma\left(j+\frac{\gamma}{2}\right)}{\Gamma\left(1+j-\frac{\gamma}{2}\right)}.
        \end{aligned}
    \end{equation}
    Finally we compute the remaining terms. We affirm that $s_1=-s_5$. In fact, by computing $s_1$ we arrive at
    \begin{equation}\label{s-4}
        \begin{aligned}
            s_1&=2^{1-\gamma}\left(1-\frac{\gamma}{2}\right) \frac{C_\gamma}{2\pi \rho_i^\gamma} \frac{2 j}{2-\gamma} b_j \frac{ \pi \cos (j \pi) \Gamma(3-\gamma)}{2^{1-\gamma} \Gamma\left(2+j-\frac{\gamma}{2}\right) \Gamma\left(2-j-\frac{\gamma}{2}\right)}\\
            &=\frac{b_j}{\rho_i^\gamma}2^{1-\gamma}\left(\frac{2-\gamma}{2}\right)\frac{2 j}{2-\gamma}  \frac{1}{2\pi}\frac{\Gamma (\frac{\gamma}{2})}{2^{1-\gamma} \Gamma(1-\frac{\gamma}{2})}  \frac{ \pi (-1)^j \Gamma(3-\gamma)}{2^{1-\gamma} \Gamma\left(2+j-\frac{\gamma}{2}\right) \Gamma\left(2-j-\frac{\gamma}{2}\right)}\\
            &=j\frac{b_j}{\rho_i^\gamma}2^{\gamma-1}\frac{1}{2}\frac{\Gamma (\frac{\gamma}{2})}{ \Gamma(1-\frac{\gamma}{2})}  \frac{  (-1)^j \Gamma(3-\gamma)}{ \Gamma\left(2+j-\frac{\gamma}{2}\right) \Gamma\left(2-j-\frac{\gamma}{2}\right)}\\
            &=j\frac{b_j}{\rho_i^\gamma}2^{\gamma-1}\frac{1}{2}\frac{(-1)^j(2-\gamma)(1-\gamma)\Gamma(1-\gamma)\Gamma (\frac{\gamma}{2})}{(1+j-\frac{\gamma}{2}) \Gamma(1-\frac{\gamma}{2})\Gamma(1+j-\frac{\gamma}{2})\Gamma\left(2-j-\frac{\gamma}{2}\right)} \\
            &=j\frac{b_j}{\rho_i^\gamma}2^{\gamma-1}\frac{\Gamma(1-\gamma)}{\left(\Gamma\left(1-\frac{\gamma}{2}\right)\right)^2} \frac{\Gamma\left(1+\frac{\gamma}{2}\right)}{\Gamma\left(2-\frac{\gamma}{2}\right)}\left( \frac{(-1)^j}{2}\frac{(2-\gamma)(1-\gamma)\Gamma(1-\frac{\gamma}{2})\Gamma(\frac{\gamma}{2})}{(1+j-\frac{\gamma}{2})(1-j-\frac{\gamma}{2})\Gamma(j+\frac{\gamma}{2})\Gamma(1-j-\frac{\gamma}{2})}\right)\\
            &=j\frac{b_j}{\rho_i^\gamma}2^{\gamma-1}\frac{\Gamma(1-\gamma)}{\left(\Gamma\left(1-\frac{\gamma}{2}\right)\right)^2} \frac{\Gamma\left(1+\frac{\gamma}{2}\right)}{\Gamma\left(2-\frac{\gamma}{2}\right)}\left( \frac{(-1)^j}{2}\frac{(2-\gamma)(1-\gamma)\left(\frac{\pi}{\sin(\pi\frac{\gamma}{2})}\right)}{(1+j-\frac{\gamma}{2})(1-j-\frac{\gamma}{2})\left(\frac{\pi}{\sin(\pi(j+\frac{\gamma}{2}))}\right)}\right)\\
            &=j\frac{b_j}{\rho_i^\gamma}2^{\gamma-1}\frac{\Gamma(1-\gamma)}{\left(\Gamma\left(1-\frac{\gamma}{2}\right)\right)^2} \frac{\Gamma\left(1+\frac{\gamma}{2}\right)}{\Gamma\left(2-\frac{\gamma}{2}\right)}\left( \frac{(-1)^j}{2}\frac{(2-\gamma)(1-\gamma)\cos(\pi j)}{(1+j-\frac{\gamma}{2})(1-j-\frac{\gamma}{2})}\right)\\
            &=j\frac{b_j}{\rho_i^\gamma}2^{\gamma-1}\frac{\Gamma(1-\gamma)}{\left(\Gamma\left(1-\frac{\gamma}{2}\right)\right)^2} \frac{\Gamma\left(1+\frac{\gamma}{2}\right)}{\Gamma\left(2-\frac{\gamma}{2}\right)}\left( \frac{1}{2}\frac{(2-\gamma)(1-\gamma)}{(1+j-\frac{\gamma}{2})(1-j-\frac{\gamma}{2})}\right).
        \end{aligned}
    \end{equation}

    A simplification of the last term $s_5$  is obtained by
    \begin{equation}\label{s-5}
        \begin{aligned}
            s_5&=  -j\frac{b_j}{\rho_i^\gamma} 2^{\gamma-1} \frac{1}{2}  \frac{\Gamma (\frac{\gamma}{2})}{  \Gamma(1-\frac{\gamma}{2})} \frac{ \Gamma(3-\gamma)}{\Gamma\left(\frac{\gamma}{2}-1\right) \Gamma\left(2-\frac{\gamma}{2}\right)}\frac{\Gamma\left(j-1+\frac{\gamma}{2}\right)}{\Gamma\left(2+j-\frac{\gamma}{2}\right)}\\
            &= -j\frac{b_j}{\rho_i^\gamma} 2^{\gamma-1} \frac{1}{2}  \frac{\Gamma (\frac{\gamma}{2})}{(j-1+\frac{\gamma}{2})  \Gamma(1-\frac{\gamma}{2})} \frac{ (2-\eta)(1-\eta)\Gamma(1-\gamma)}{\Gamma\left(\frac{\gamma}{2}-1\right) \Gamma\left(2-\frac{\gamma}{2}\right)}\frac{\Gamma\left(j+\frac{\gamma}{2}\right)}{\left(1+j-\frac{\gamma}{2}\right)\Gamma\left(1+j-\frac{\gamma}{2}\right)}\\
            &=j \frac{b_j}{\rho_i^\gamma} 2^{\gamma-1} \frac{\Gamma(1-\gamma)}{\left(\Gamma\left(1-\frac{\gamma}{2}\right)\right)^2} \frac{\Gamma\left(j+\frac{\gamma}{2}\right)}{\Gamma\left(1+j-\frac{\gamma}{2}\right)}\left(\frac{-1}{2}\frac{(2-\gamma)(1-\gamma)\Gamma(\frac{\gamma}{2})}{(j-1+\frac{\gamma}{2})(1+j-\frac{\gamma}{2})(1-\frac{\gamma}{2})\Gamma(\frac{\gamma}{2}-1)}\right) \\
            &=j \frac{b_j}{\rho_i^\gamma} 2^{\gamma-1} \frac{\Gamma(1-\gamma)}{\left(\Gamma\left(1-\frac{\gamma}{2}\right)\right)^2} \frac{\Gamma\left(j+\frac{\gamma}{2}\right)}{\Gamma\left(1+j-\frac{\gamma}{2}\right)}\left(\frac{1}{2}\frac{(2-\gamma)(1-\gamma)\Gamma(\frac{\gamma}{2})}{(j-1+\frac{\gamma}{2})(1+j-\frac{\gamma}{2})(\frac{\gamma}{2}-1)\Gamma(\frac{\gamma}{2}-1)}\right) \\
            &=j \frac{b_j}{\rho_i^\gamma} 2^{\gamma-1} \frac{\Gamma(1-\gamma)}{\left(\Gamma\left(1-\frac{\gamma}{2}\right)\right)^2} \frac{\Gamma\left(j+\frac{\gamma}{2}\right)}{\Gamma\left(1+j-\frac{\gamma}{2}\right)}\left(\frac{-1}{2}\frac{(2-\gamma)(1-\gamma)}{(1-j-\frac{\gamma}{2})(1+j-\frac{\gamma}{2})}\right) ,
        \end{aligned}
    \end{equation}
    where we used the fact that  $\Gamma(z)\Gamma(1-z )=\frac{\pi}{\sin(\pi z)}$ and $\Gamma(z+1)=z\Gamma(z)$ in the above computations. Hence we have $s_1=-s_5$. Then, the above sum has the following form
    \begin{equation}\label{s-6}
        \begin{aligned}
            \sum_{k=1}^{5} s_k&=j \frac{b_j}{\rho_i^\gamma} 2^{\gamma-1}  \frac{ \Gamma(1-\gamma)}{ \left(\Gamma\left(1-\frac{\gamma}{2}\right)\right)^2}\frac{\Gamma\left(j+\frac{\gamma}{2}\right)}{\Gamma\left(1+j-\frac{\gamma}{2}\right)}-j \frac{b_j}{\rho_i^\gamma} 2^{\gamma-1} \frac{\Gamma(1-\gamma)}{\left(\Gamma\left(1-\frac{\gamma}{2}\right)\right)^2} \frac{\Gamma\left(1+\frac{\gamma}{2}\right)}{\Gamma\left(2-\frac{\gamma}{2}\right)}\\
            &=-j \frac{b_j}{\rho_i^\gamma} 2^{\gamma-1}\frac{ \Gamma(1-\gamma)}{(\Gamma(1-\frac{\gamma}{2}))^2}\left(\frac{\Gamma(1+\frac{\gamma}{2})}{\Gamma(2-\frac{\gamma}{2})}-\frac{\Gamma(j+\frac{\gamma}{2})}{\Gamma(1+j-\frac{\gamma}{2})}\right).
        \end{aligned}
    \end{equation}
Summing it up, the $j$-th coefficient is exactly
    \begin{equation}\label{eqlin}
        \partial_{g_i} G_i(0, \boldsymbol{\rho}, \boldsymbol{x}, \mathbf{0})(\sin (j \beta))=-\frac{b_j}{ \rho_i^\gamma}   j \sigma_j\cos (j \beta) .
    \end{equation}
    Then, we have
    \begin{equation*}
    \partial_{g_i} G_i(0, \boldsymbol{\rho}, \boldsymbol{x}, \mathbf{0}) h=\frac{1}{ \rho_i^\gamma} \sum_{j=1}^{\infty} \sigma_j j\left(a_j \sin (j \beta)-b_j \cos (j \beta)\right).
    \end{equation*}
    To establish the isomorphism of $\partial_{g_i}G_{i}(0, \boldsymbol{\rho}, \boldsymbol{x}, \mathbf{0})h:X^{k+\gamma-1}\to Y_0^{k-1}$, it is observed that the sequence ${\sigma_j}$, as demonstrated in \cite[Proposition 2.7]{Cas1}, is monotonically increasing and possesses a positive lower bound. Consequently, the kernel of $\partial_{g_i}G_{i}(0, \boldsymbol{\rho}, \boldsymbol{x}, \mathbf{0})$ is trivial. Our goal is now to demonstrate that for any $p(x)\in Y_0^{k-1}$, there exists an $h(x) \in X^{k+\gamma-1}$ such that $\partial_{g_i}G_{i}(0, \boldsymbol{\rho}, \boldsymbol{x}, \mathbf{0})h=p$. Referring back to the first part of this lemma, if $p$ can be represented as $p(x)=\sum\limits_{j=1}^\infty c_j\sin(jx)+d_j\cos(jx)$, then $h$ must satisfy the following
	\begin{equation*}
		h(x)=r^\gamma_i\sum\limits_{j=1}^\infty \left(c_j\sigma_j^{-1}j^{-1}\cos(jx)+d_j\sigma_j^{-1}j^{-1}\sin(jx)\right).
	\end{equation*}
    Utilizing the asymptotic expansion of the Gamma function, we derive the following asymptotic behavior: $\sigma_j=O(j^{\gamma-1})$ for $\gamma\in(1,2)$ (as detailed in \cite[Lemma 2.8]{Cas1}). Consequently, we can infer the following
    \begin{equation*}
        \begin{aligned}
            \|h\|_{X^{k+\gamma-1}}&=r^\gamma_i\sum\limits_{j=1}^\infty (c_j^2+d_j^2)\sigma_j^{-2}j^{-2}j^{2k+2\gamma-2}\\
            &\leq  C\sum\limits_{j=1}^\infty (c_j^2+d_j^2)j^{-2(\gamma-1)}j^{-2}j^{2k+2\gamma-2}\\
            &\leq  C\sum\limits_{j=1}^\infty (c_j^2+d_j^2)j^{2k-2}\le C\|p\|_{Y_0^{k-1}},
        \end{aligned}
    \end{equation*}
    and the proof of the first part is now finished.

    On the other hand,  by applying  \eqref{4-1}, we have $\partial_{g_j} G_i(0, \boldsymbol{\rho}, \boldsymbol{x}, \mathbf{0}) h=0, j \neq i$. Therefore, we find
    \begin{equation}\label{diag}
        \partial_{\boldsymbol{g}} \boldsymbol{G}\left(0, \boldsymbol{\rho},\boldsymbol{x}, 0\right)=\operatorname{diag}\left(\partial_{g_1} G_1\left(0, \boldsymbol{\rho},\boldsymbol{x},0\right), \cdots, \partial_{g_m} G_m\left(0, \boldsymbol{\rho},\boldsymbol{x}, 0\right)\right)
    \end{equation}
    and hence $\partial_{\boldsymbol{g}} \boldsymbol{G}(0, \boldsymbol{\rho}, \boldsymbol{x}, \mathbf{0})$ is also an isomorphism from $(X^{k+\gamma-1})^m$ to $ (Y_0^{k-1})^m$.

    Hence, the proof is now concluded.
\end{proof}

The subsequent step involves solving the equation $F_i=0$ defined in \eqref{Gi_eq} by appropriately selecting the values of $\rho_i$.
\begin{lemma}\label{r}
  There exists a positive value $\varepsilon_0$ such that, for any $\varepsilon$ within the range $\left(-\varepsilon_0, \varepsilon_0\right)$ and for arbitrary functions $g_i$ satisfying the condition
$\norm{g_i}_{X^{k+\gamma-1}}<1$, the equation $$F_i(\varepsilon,\rho_i,g_i)=0,$$ possesses a unique solution given by
    \begin{equation*}
        \rho_i=\rho_i\left(\varepsilon, g_i\right)=\sqrt{\frac{\kappa_i}{\pi}}+\varepsilon^{2+2\gamma} \mathcal{R}\left(\varepsilon, g_i\right) .
    \end{equation*}
    In this context, $\mathcal{R}\left(\varepsilon, g_i\right)$ represents a bounded functional that is also continuously differentiable.
\end{lemma}

\begin{proof}
First, note that for $\abs{z}<1$
    \begin{equation}
        \begin{split}
            (1+z)^{1/2}&=\sum_{n=0}^{\infty}\binom{1/2}{n}z^n\\
            &
            =1+\frac {z}{2}+\frac {1/2 (1/2-1)}{2!}z^2+
\cdots +\frac {1/2 (1/2-1)\cdots (1/2-n+1)}{n!}z^n+\cdots
        \end{split}
    \end{equation}
Now, for $\frac{\kappa_i}{\pi}\neq 0$,   the expansion is given by
$$\left(\frac{\kappa_i}{\pi}- \frac{\varepsilon^{2+2\gamma}}{2\pi} \int_0^{2 \pi} g_i^2(\beta) d \beta\right)^{1/2}=r(1+z)^{1/2}=\sqrt{\frac{\kappa_i}{\pi}}(1+z)^{1/2}$$
with $z=-\varepsilon^{2+2\gamma}\frac{b }{2\kappa_i} \int_0^{2 \pi} g_i^2(\beta) d \beta$
and where $r$ is the principal square root of $\frac{\kappa_i}{\pi}$. Since we are dealing with real numbers and $\frac{\kappa_i}{\pi}>0$ then we have $r=\sqrt{\frac{\kappa_i}{\pi}}$.

Hence, it is straightforward to observe that equation (\ref{Gi_eq}) is equivalent to
    \begin{equation*}
        \rho_i\left(\varepsilon, g_i\right)=\sqrt{\frac{\kappa_i}{\pi}-\frac{\varepsilon^{2+2\gamma}}{2\pi} \int_0^{2 \pi} g_i^2(\beta) d \beta}=\sqrt{\frac{\kappa_i}{\pi}}+\varepsilon^{2+2\gamma} \mathcal{R}\left(\varepsilon, g_i\right).
    \end{equation*}
\end{proof}
\section{Existence of vortex patches}\label{section5}

Within this section, taking inspiration from the classical Crandall-Rabinowitz theorem within the realm of bifurcation theory. We construct a family of time-periodic vortex patch solution to \eqref{1-1} with $\gamma\in(1,2)$, for any arbitrarily chosen small value of $\varepsilon$  by  applying implicit function theorem.

Define $\boldsymbol{\rho}(\varepsilon, \boldsymbol{g}):=\left(\rho_1\left(\varepsilon, g_1\right), \ldots, \rho_m\left(\varepsilon, g_m\right)\right)$, where $\rho_i\left(\varepsilon, g_i\right)$ is determined by Lemma \ref{r}. Now, define the new functional $\mathcal{G}$ as follows
 \begin{equation*}
    \mathcal{G}(\varepsilon, \boldsymbol{x}, \boldsymbol{g})=\boldsymbol{G}(\varepsilon, \boldsymbol{\rho}(\varepsilon, \boldsymbol{g}), \boldsymbol{x}, \boldsymbol{g}).
\end{equation*}

In order to apply the implicit function theorem, it is essential to ensure that $\mathcal{G}$ maps a suitable subset of $(X^{k+\gamma-1})^m$ into $(Y_0^k)^m$. This goal can be achieved by making a proper choice for $\boldsymbol{x}.$ In fact, by letting
\begin{equation*}
    V_1:=\left\{\boldsymbol{g} \in (X^{k+\gamma-1})^m \mid \sum_{j=1}^m\left\|g_i\right\|_{H^{k}}<1\right\} \subset (X^{k+\gamma-1})^m
\end{equation*}
be the unit ball, and with this choice, we can establish the following crucial lemma.

\begin{lemma}\label{zeros}
   The requirement that $\mathcal{G}$ maps $\left(-\varepsilon_0, \varepsilon_0\right) \times B_{\rho_0}\left(\boldsymbol{x}_0\right) \times V_1$ into $(Y_0^{k-1})^m$ can be alternatively expressed as a system of $2 m$ equations of the following type
    \begin{equation}\label{w}
        \nabla \mathcal{W}_m(\boldsymbol{x})=O(\varepsilon).
    \end{equation}
   In this context, the notation $O(\varepsilon)$ represents a term that exhibits $C^1$ smoothness and is uniformly bounded by $C \varepsilon$, where the constant $C$ is independent of both $\varepsilon$ and $\boldsymbol{g}$
\end{lemma}

\begin{proof}
    For any given $i=1, \ldots, m$, we choose $\boldsymbol{g} \in V_1$ such that
    \begin{equation*}
    g_i(\beta)=\sum_{j=2}^{\infty}\left(a_j\cos(j\beta)+b_j\sin(j\beta)\right).
    \end{equation*}
    According to the definition of $(Y_0^{k-1})^m$, to guarantee that $\mathcal{G}(\varepsilon, \boldsymbol{x}, \boldsymbol{g})=\boldsymbol{G}(\varepsilon, \boldsymbol{\rho}(\varepsilon, \boldsymbol{g}), \boldsymbol{x}, \boldsymbol{g}) \in (Y_0^{k-1})^m$ is equivalent to prove that  $\mathcal{G}_i(\varepsilon, \boldsymbol{x}, \boldsymbol{g})=G_i(\varepsilon, \boldsymbol{\rho}(\varepsilon, \boldsymbol{g}), \boldsymbol{x}, \boldsymbol{g}) \in Y_0^{k-1}$, for all $i=1,\cdots,m$. Hence, we must ensure that the following equations are satisfied
    \begin{equation}\label{ff}
        \begin{aligned}
            & \int_0^{2 \pi} G_i(\varepsilon,\boldsymbol{\rho}(\varepsilon, \boldsymbol{g}),\boldsymbol{x},\boldsymbol{g})\sin(\beta) d\beta=0, \\
            & \int_0^{2 \pi} G_i(\varepsilon, \boldsymbol{\rho}(\varepsilon, \boldsymbol{g}), \boldsymbol{x}, \boldsymbol{g}) \cos (\beta) d \beta=0,
        \end{aligned}
    \end{equation}
    where $i=1, \ldots, m$. Using equation \eqref{3-7}, along with the calculations from Lemma \ref{iso} and Lemma \ref{r}, we deduce
    \begin{equation}
        \begin{aligned}\label{fff}
            & \int_0^{2 \pi} G_i(\varepsilon, \boldsymbol{\rho}(\varepsilon, \boldsymbol{g}), \boldsymbol{x}, \boldsymbol{g}) \sin (\beta) d \beta \\
            = & \sum_{j \neq i} \kappa_j \frac{\gamma C_\gamma\left(x_i-x_j\right) \cdot e_1}{2 \pi\left|x_i-x_j\right|^{\gamma+2}}+\sum_{j=1}^m \kappa_j e_1 \cdot \nabla_x K_\gamma^0\left(x_i, x_j\right)+O(\varepsilon)
        \end{aligned}
    \end{equation}
    and
    \begin{equation}\label{ffff}
        \begin{aligned}
            & \int_0^{2 \pi} G_i(\varepsilon,\boldsymbol{\rho}(\varepsilon, \boldsymbol{g}), \boldsymbol{x}, \boldsymbol{g}) \cos (\beta) d \beta \\
            = & -\sum_{j \neq i} \kappa_j \frac{\gamma C_\gamma\left(x_i-x_j\right) \cdot e_2}{2 \pi\left|x_i-x_j\right|^{\gamma+2}}+\sum_{j=1}^m \kappa_j e_2 \cdot \nabla_x K_\gamma^0\left(x_i, x_j\right)+O(\varepsilon) ,
        \end{aligned}
    \end{equation}
    with $e_1=(1,0)$ and $e_2=(0,1)$, the equations $\eqref{fff}$ and $\eqref{ffff}$ can be understood as being equivalent to equation $\eqref{w}$, thanks to the definition of $\mathcal{W}_m$.
\end{proof}

With all the necessary preparations in the preceding lemmas, we are now in a position to prove our main theorem.

\begin{proof}[Proof of Theorem \ref{thm1}]
    Given the non-degeneracy condition $\operatorname{deg}\left(\nabla \mathcal{W}_{m}, \boldsymbol{x}_0\right) \neq 0$, equation \eqref{w} possesses a unique solution near $\boldsymbol{x}_0$ for sufficiently small $\varepsilon$, as established by the implicit function theorem. We proceed to solve \eqref{w} and express the solution $\boldsymbol{x}_{\varepsilon}$ in the form $\boldsymbol{x}_{\varepsilon}=\boldsymbol{x}_0+\varepsilon \mathcal{R}_{\boldsymbol{x}}(\varepsilon, \boldsymbol{g})$. We know that $\mathcal{R}_{\boldsymbol{x}}$ defined on $\left(-\varepsilon_0, \varepsilon_0\right) \times V_1$ is at least $C^1$ due to the regularity of the functional $\boldsymbol{G}=\left(G_1, \ldots, G_m\right)$.

    Now, we introduce the following notation
    \begin{equation*}
        \bar{G}_i(\varepsilon, \boldsymbol{g}):=G_i\left(\varepsilon, \boldsymbol{\rho}(\varepsilon, \boldsymbol{g}), \boldsymbol{x}_0+\varepsilon \mathcal{R}_{\boldsymbol{x}}(\varepsilon, \boldsymbol{g}), \boldsymbol{g}\right)
    \end{equation*}
    and define
    \begin{equation*}
        \overline{\boldsymbol{\mathcal{G}}}(\varepsilon, \boldsymbol{g})=\left(\bar{G}_1(\varepsilon, \boldsymbol{g}), \ldots, \bar{G}_m(\varepsilon, \boldsymbol{g})\right) .
    \end{equation*}
    We can then conclude from Lemma \ref{zeros} that $\overline{\boldsymbol{\mathcal{G}}}$ maps $\left(-\varepsilon_0, \varepsilon_0\right) \times V_1$ into $\left(Y_0^{k-1}\right)^m$. Furthermore, $\overline{\boldsymbol{\mathcal{G}}}$ is $C^1$ continuous with respect to $\boldsymbol{g}$. To proceed, we need to verify that $\partial_{\boldsymbol{g}}\overline{\boldsymbol{\mathcal{G}}}(0,\boldsymbol{0})$ is an isomorphism from $\left(X^{k+\gamma-1 }\right)^m$ to $\left(Y_0^{k-1}\right)^m$. By invoking the chain rule, we obtain
    \begin{equation*}
        \begin{aligned}
            & \partial_{g_j} \bar{G}_i(\varepsilon, \boldsymbol{g})=\partial_{g_j} G_i\left(\varepsilon, \boldsymbol{\rho}(\varepsilon, \boldsymbol{g}), \boldsymbol{x}_0+\varepsilon \mathcal{R}_{\boldsymbol{x}}(\varepsilon, \boldsymbol{g}), \boldsymbol{g}\right)+\partial_{\boldsymbol{\rho}} G_i\left(\varepsilon, \boldsymbol{\rho}(\varepsilon, \boldsymbol{g}), \boldsymbol{x}_0\right. \\
            & \left.\quad+\varepsilon \mathcal{R}_{\boldsymbol{x}}(\varepsilon, \boldsymbol{g}), \boldsymbol{g}\right) \cdot \partial_{\boldsymbol{g}_j} \boldsymbol{\rho}(\varepsilon, \boldsymbol{g})+\partial_{\boldsymbol{x}} G_i\left(\varepsilon, \boldsymbol{\rho}(\varepsilon, \boldsymbol{g}), \boldsymbol{x}_0+\varepsilon \mathcal{R}_{\boldsymbol{x}}(\varepsilon, \boldsymbol{g}), \boldsymbol{g}\right) \cdot \partial_{g_j}\left(\varepsilon \mathcal{R}_{\boldsymbol{x}}(\varepsilon, \boldsymbol{g})\right),
        \end{aligned}
    \end{equation*}
    which implies
    \begin{equation*}
        \partial_{g_j} \bar{G}_i(0,0)=\partial_{g_j} G_i\left(\varepsilon, \boldsymbol{\rho}(0,\boldsymbol{0}), \boldsymbol{x_0}, 0\right).
    \end{equation*}
    Hence, by Lemma \ref{iso}, $\partial_{\boldsymbol{g}}\overline{\boldsymbol{\mathcal{G}}}(0,\boldsymbol{0})$ is an isomorphism from $\left(X^{k+\gamma-1 }\right)^m$ to $\left(Y_0^{k-1}\right)^m$. It is important to note that Lemma \ref{zeros} implies $\overline{\boldsymbol{\mathcal{G}}}(0,\boldsymbol{0})=0$.

  Through the application of the implicit function theorem to  $\overline{\boldsymbol{\mathcal{G}}}$ at the point $(0,0)$, we can deduce the existence of $\varepsilon_0>0$. Consequently, the solution set
    \begin{equation*}
        \left\{(\varepsilon, \boldsymbol{g}) \in\left(-\varepsilon_0, \varepsilon_0\right) \times V_1 \mid \overline{\boldsymbol{\mathcal{G}}}(\varepsilon, \boldsymbol{g})=0\right\}
    \end{equation*}
    is not empty and can be parameterized by one-dimensional curve  $\varepsilon \in\left(-\varepsilon_0, \varepsilon_0\right) \rightarrow\left(\varepsilon, \boldsymbol{g}_{\varepsilon}\right)$. Consequently, we have established a family of non-trivial vortex patch solutions with fixed vorticity $\frac{1}{\varepsilon^2}$ and total flux $\kappa_i$ for each patch by using Lemma \ref{r}, fulfilling conditions $(i)$-$(iii)$ of Theorem \ref{thm1}.

    Since $(iv)$ of Theorem \ref{thm1} is trivial, to end our proof we only need to establish the convexity of the interior of $\Gamma_i$ for $i=1, \ldots, m$.  To achieve this, our objective is to compute the signed curvature $C_i(\beta)$  of the patch $\Gamma_i$ at $z_i(\beta)=(z_{i}^{1}(\beta),z_{i}^{2}(\beta))=x_{ i}+\varepsilon R_i(\beta)(\cos \beta, \sin \beta)$ with $R_i(\beta)=$ $1+\varepsilon^{1+\gamma}g_{i}(\beta)$ for  $\beta \in[0,2 \pi)$. Specifically, we have
    \begin{equation*}
        \begin{split}
            \varepsilon  C_{i}(\beta)& =\frac{\partial _{\beta\beta}z_{i}^{2}(\beta)\partial_{\beta}z_{i}^{1}(\beta)-\partial _{\beta\beta}z_{i}^{1}(\beta)\partial _{\beta}z_{i}^{2}(\beta)}{((\partial _{\beta}z_{i}^{1}(\beta))^2+(\partial_{\beta}z_{i}^{2}(\beta))^2)^{3/2}} \\
            &=\frac{(1+\varepsilon^{1+\gamma}g_{i}(\beta))^{2}+2\varepsilon^{2+2\gamma}(g_{i}^{\prime }(\beta))^{2}-\varepsilon^{1+\gamma}g_{i}^{\prime\prime }(\beta)(1+\varepsilon^{1+\gamma}g_{i}(\beta))}{\left( (1+\varepsilon^{1+\gamma}g_{i}(\beta))^{2}+\varepsilon^{2+2\gamma}(g_{i}^{\prime}(\beta))^{2}\right) ^{\frac{3}{2}}}=\frac{1+O(\varepsilon )}{1+O(\varepsilon )}>0,
        \end{split}%
    \end{equation*}
    for small $\varepsilon $ and each $\beta\in \lbrack 0,2\pi )$. The resulting quantity is non-negative if $\varepsilon \in (-\varepsilon_0,\varepsilon_0)$. Consequently, the signed curvature is strictly positive, confirming the desired result $(v)$. Therefore, the proof of Theorem \ref{thm1} is now complete.
\end{proof}

\section*{Acknowledgement}


E. Cuba has been supported by FAPESP through grant  2021/10769-6, Brazil. L. C. F. Ferreira has been partially supported by FAPESP and CNPq, Brazil.

\phantom{s} \thispagestyle{empty}

\end{document}